\documentclass[12pt]{amsart}

\usepackage{amsmath}
\usepackage{amssymb}
\usepackage[curve]{xypic}
\usepackage{caption}
\usepackage{xspace}
\usepackage{cite}
\usepackage{enumitem}
\usepackage{url}
\usepackage{tikz}
\usepackage{color}
\usepackage{xcolor}
\usetikzlibrary{patterns}
\usepackage[section]{placeins}

\makeatletter 
\def\@cite#1#2{{\m@th\upshape\bfseries%
[{#1\if@tempswa{\m@th\upshape\mdseries, #2}\fi}]}}
\makeatother 

\theoremstyle{plain}
\newtheorem{thm}{Theorem}[section]
\newtheorem{cor}[thm]{Corollary}
\newtheorem{prop}[thm]{Proposition}
\newtheorem{lem}[thm]{Lemma}

\theoremstyle{definition}

\newtheorem{prob}[thm]{Problem}

\theoremstyle{remark}
\newtheorem{rem}[thm]{Remark}

\numberwithin{equation}{subsection}
\renewcommand{\bold}[1]{\medskip \noindent {\bf #1 }\nopagebreak}

\newcommand{\nc}{\newcommand}
\newcommand{\rnc}{\renewcommand}
\newcommand{\e}{\varepsilon}

\newcommand{\fp}[1]{\langle\langle #1 \rangle\rangle}

\nc\bA{\mathbb{A}}
\nc\bB{\mathbb{B}}
\nc\bC{\mathbb{C}}
\nc\bD{\mathbb{D}}
\nc\bE{\mathbb{E}}
\nc\bF{\mathbb{F}}
\nc\bG{\mathbb{G}}
\nc\bH{\mathbb{H}}
\nc\bI{\mathbb{I}}
\nc{\bJ}{\mathbb{J}} 
\nc\bK{\mathbb{K}}
\nc\bL{\mathbb{L}}
\nc\bM{\mathbb{M}}
\nc\bN{\mathbb{N}}
\nc\bO{\mathbb{O}}
\nc\bP{\mathbb{P}}
\nc\bQ{\mathbb{Q}}
\nc\bR{\mathbb{R}}
\nc\bS{\mathbb{S}}
\nc\bT{\mathbb{T}}
\nc\bU{\mathbb{U}}
\nc\bV{\mathbb{V}}
\nc\bW{\mathbb{W}}
\nc\bY{\mathbb{Y}}
\nc\bX{\mathbb{X}}
\nc\bZ{\mathbb{Z}}
\nc\cA{\mathcal{A}}
\nc\cB{\mathcal{B}}
\nc\cC{\mathcal{C}}
\rnc\cD{\mathcal{D}}
\nc\cE{\mathcal{E}}
\nc\cF{\mathcal{F}}
\nc\cG{\mathcal{G}}
\rnc\cH{\mathcal{H}}
\nc\cI{\mathcal{I}}
\nc{\cJ}{\mathcal{J}} 
\nc\cK{\mathcal{K}}
\rnc\cL{\mathcal{L}}
\nc\cM{\mathcal{M}}
\nc\cN{\mathcal{N}}
\nc\cO{\mathcal{O}}
\nc\cP{\mathcal{P}}
\nc\cQ{\mathcal{Q}}
\rnc\cR{\mathcal{R}}
\nc\cS{\mathcal{S}}
\nc\cT{\mathcal{T}}
\nc\cU{\mathcal{U}}
\nc\cV{\mathcal{V}}
\nc\cW{\mathcal{W}}
\nc\cY{\mathcal{Y}}
\nc\cX{\mathcal{X}}
\nc\cZ{\mathcal{Z}}

\nc{\LS}{\ensuremath{\underset{n=1}{\overset{\infty}{\cap}} \, {\underset{i=n}{\overset{\infty}{\cup}}}\,}}

\nc{\dmo}{\DeclareMathOperator}
\dmo{\Area}{Area}
\dmo{\Var}{Var}

\begin{document}

\title[A mixing flow on a surface]{A mixing flow on a surface with non-degenerate fixed points}
%

\author[Chaika]{Jon~Chaika}
\address{\hspace{-0.5cm}Department of Mathematics\newline
University of Utah\newline
155 S 1400 E, Room 233\newline
Salt Lake City, UT 84112}
\email{chaika@math.utah.edu}
\author[Wright]{Alex~Wright}
\address{\hspace{-0.5cm}Department of Mathematics\newline
Stanford University\newline
Palo Alto, CA 94305}
\email{amwright@stanford.edu}
%

\begin{abstract}
We construct a smooth, area preserving, mixing flow with finitely many non-degenerate fixed points and no saddle connections on a closed surface of genus 5. This resolves a problem that has been open for four decades.  
\end{abstract}

\maketitle
\thispagestyle{empty}


\newenvironment{xquote}
  {\list{}{\leftmargin=0pt\rightmargin\leftmargin}\item\relax}
  {\endlist}

\section{Introduction}\label{S:intro}

\bold{Motivation and main result.} Flows on surfaces are a basic example in smooth dynamics, being in a sense the smallest smooth dynamical systems after circle diffeomorphisms, and have been the topic of a vast body of research. 

In 1972, the existence of smooth ergodic flows on all closed surfaces except the sphere, projective plane, and Klein bottle was established \cite{Blo}. Only a few years later Katok, Sinai, and Stepin indicated the following as an open problem in their 1975 survey paper \cite[4.4.1]{KSS}.

 ``Let $T_t$ be a smooth flow on a surface of genus $p\geq2$ with smooth positive invariant measure, all of whose fixed points are non-degenerate saddles. Can $T_t$ be mixing? The distinguished results of A. V. Kochergin and A. B. Katok give a negative answer to this question in all probability."

The same question was listed by Katok and Thouvenot  in the Handbook of Dynamical Systems \cite[Problem 6.10]{kat_thouv} and was mentioned by Forni in \cite[Page 4]{forni_dev}. The purpose of this paper is to provide a positive answer. 

\begin{thm}\label{T:main}
There is a mixing, smooth, area preserving flow on a surface of genus 5 with finitely many fixed points, all non-degenerate, and no saddle connections. 
\end{thm}

A saddle connection is a flow trajectory joining two fixed points for the flow. The derivative of a smooth flow is a vector field, which can be written locally on a surface as $A(x,y) \partial_x + B(x,y) \partial_y$. A fixed point is called non-degenerate if at that point the function $(A,B)$ has non-zero Jacobian, i.e. if $A_x B_y-A_y B_x\neq 0$. 

Already at the time of Katok, Sinai, and Stepin's question, Kochergin had shown that mixing flows do exist on surfaces if degenerate saddles are allowed \cite{koc_mix}.  However, the presence of degenerate saddles has such a drastic effect that it is reasonable to believe that the natural class of dynamical systems that should be grouped together is not all smooth area preserving flows on surfaces, but rather those with only finitely many saddles, all of which are non-degenerate. 

Kochergin's result that mixing flows exist in fact supports making a large distinction between degenerate and non-degenerate saddles, since he produced mixing flows with degenerate fixed points even on tori, but later went on to show that flows on tori with only non-degenerate fixed points are never mixing \cite{koc2007}. 

The intuition that the types of flows considered in our main theorem are very unlikely to be mixing has proven correct, as Ulcigrai has recently established that such flows are generically not mixing \cite{ulc_sym}, following work of Scheglov showing this in genus 2 \cite{scheg}. It is however known that these flows are generically uniquely ergodic \cite{Merg, Verg} and weak mixing \cite{ulc_wmix}. (The notion of generic here is measure theoretic.) Many examples are mild mixing \cite{KKP}. 

\bold{Kochergin's mechanism for mixing.} Consider a small horizontal line segment in $\bR^2$. Under the action of $\left(\begin{array}{cc} 1&0\\t&1 \end{array}\right)$ this small horizontal line segment will be sheared until eventually it is close to a long vertical line segment. Similarly, an interval transverse to a flow may eventually get sheared so much that it becomes close to an orbit of the flow. Kochergin's seminal observation is that in this case, if the flow is ergodic, one may expect equidistribution of such flowed transverse intervals, and subsequently hope to conclude that the flow is mixing.

This idea has been used in many subsequent works, and has also been applied to flows on higher dimensional manifolds, such as Fayad's example of a reparameterization of a linear flow on $\bT^3$ that is mixing and has singular spectrum \cite{fayad}. (We mention in passing that mixing is easier to obtain in dimension greater  two, and indeed Fayad obtains mixing for a flow without fixed points, because there are two dimensions transverse to the flow which may alternately be sheared.) 

Kochergin's technique will be the engine of mixing in Theorem \ref{T:main}. The shearing effects are most significant near the fixed points of the flow, since trajectories that pass closer to a fixed point will get slowed down more than trajectories that pass farther away. This is why degenerate fixed points can help in establishing mixing; they establish an extreme shearing effect. For non-degenerate fixed points, the shearing effect is less extreme, and typically  many passes near a fixed point are required to accumulate an appreciable amount of shearing. However, 
passing on different sides of a fixed point produces opposite effects, which are expected to cancel out. This leads to the intuition, which has been made rigorous with the work of Scheglov, Ulcigrai, and others, that area preserving flows should not be mixing. 

\bold{Suspension flows.} 
Given a flow on a surface with finitely many fixed points and no saddle connections, pick a disjoint union $I$ of intervals transverse to the flow direction. (Typically one picks only one interval, but it will be convenient for us to use four.) Let $T:I\to I$ be the first return map, and let $f:I\to (0,\infty]$ be the first return time function, so the flow is isomorphic to the vertical flow on the space 
$$\{(x,s):x\in I, 0\leq s\leq f(x)\}/((x,f(x)) \sim  (T(x), 0)).$$

If the flow is measure preserving, $I$ can be parameterized so that $T$ is a \emph{multi-interval exchange transformation}, i.e., a permutation of a finite number of subintervals that partition $I$. (If there is only one interval, this is called an \emph{interval exchange transformation}.)
Since the flow is smooth, $f$ is smooth away from the discontinuities of $T$. 

The standard model for a non-degenerate fixed point is given by the vector field $x \partial_y+y\partial_x$. This has two incoming trajectories, and two outgoing trajectories. 
For a flow with only finitely many non-degenerate fixed points, the first return map is, up to a bounded function with bounded derivative, equal to a function of the form $f=1-\sum c_i \log |x-x_i|$. Since the roof function is infinite at the $x_i$, these are the points that orbit into a fixed point. If $x_i$ and $x_j$ orbit into the  same discontinuity, then $c_i=c_j$. More precise statements can be found in \cite{kat_skew} and \cite[Section 3]{koc_nomix}. 

Moreover, standard arguments show that all   $T$ and $f$ satisfying certain technical conditions arise from smooth flows on surfaces with only finitely many non-degenerate fixed points \cite[Section 7]{con_fra}. Thus, to prove our main theorem, we will find  an appropriate $T$ and $f$ for which the suspension flow is mixing, and in the last section of this paper will will explain how Theorem \ref{T:main} follows. 

\bold{Birkhoff sums of non-integrable functions.} 
We will see that the net shearing of an interval transverse to the flow is controlled by Birkhoff sums of $f'$, that is, by sums of the form 
$$\sum_{i=0}^{N-1} f'(T^ix).$$
Our roof function $f$ will have $f'$ non-integrable, so the Birkhoff Ergodic Theorem may not be used to understand these sums. Note also that $f$ will not be of bounded variation. Katok has shown that suspension flows over interval exchange transformations with roof functions of bounded variation are never mixing \cite{kat_nomix}.

To get enough shearing for mixing, we will require the above  Birkhoff sums to grow faster than linearly in $N$, and we will need fairly precise control. 

The problem is that  we expect a large amount of cancellation to occur between positive and negative terms in this Birkhoff sum. When $T^i(x)$ is close to and on the right side of a singularity, $f'(T^ix)$ will be very negative. When $T^i(x)$ is close to and on the left side of a singularity,  $f'(T^ix)$ will be very positive.  

\bold{The Katok-Sataev-Veech construction.} In turns out that the following result is technically easier to prove than Theorem \ref{T:main}. 

\begin{thm}\label{T:genus2}
There is a $\mathbb{Z}_2$ skew product $T$ over a rotation with the following properties:  
 $T$ has four discontinuities, $x_1, x_2, x_3, x_4$. There is a fifth point $x_0$ that is not a discontinuity, such that the suspension flow over $T$  with roof function
$$f(x)=1-\sum_{i=0}^4 \log |x-x_i|$$
is mixing. 
\end{thm}
This choice of  $T$ and $f$ do not satisfy the technical conditions to correspond to a smooth flow, because the roof function has one extra  singularity that is not at a discontinuity of $T$. In the final section we choose a closely related but more complicated $\hat{T}$ and $\hat{f}$ to prove  Theorem \ref{T:main}.

Most of this paper is occupied with the proof of Theorem \ref{T:genus2}. 
To build $T$, we modify the Katok-Sataev-Veech construction for producing examples of minimal but non-uniquely ergodic interval exchange transformations \cite{v_skew, kat_skew, sat_skew}. Our $T$ will in fact be uniquely ergodic, but orbits equidistribute very slowly and in a controlled manner.

To obtain mixing, we construct $T$ to be very well approximated by  non-minimal $\mathbb{Z}_2$ skew products of rotations, $T_k$, such that $T_k$ has two minimal components, one of which contains an interval to the left of $x_0$, and one of which contains an interval to the right of $x_0$. Quantitative estimates, and highly non-generic choices of parameters such as the continued fraction expansion  of the base rotation, allow us to show that this asymmetry in the minimal approximates yields appropriate growth in the Birkhoff sums of $f'$ and hence obtains sufficient shearing for mixing. To get this growth, we must prevent the terms in these sums where $T^i(x)$ is to the right of $x_0$ from canceling with the terms where $T^i(x)$ is to the left of $x_0$. This is difficult because $T$ is uniquely ergodic, so all orbits equidistribute. We show that the terms where $T^i(x)$ is in certain decreasing neighborhoods of $x_0$ dominate these sums, and within these smaller and smaller neighborhoods  orbits of certain lengths are not at all equidistributed. 

Technically speaking, our analysis makes heavy use of continued fractions and the Denjoy-Koksma inequality. We chose parameters so that the base rotation has orbits which equidistribute very quickly, but the two minimal components in the skew product equidistribute slowly.

\bold{Open problems.} The proof of Theorem \ref{T:main} is built on the fact that there are non-minimal smooth flows with finitely many non-degenerate fixed points on surfaces of genus 5 such that one minimal component sees only one side of a fixed point. Such flows also exist on surfaces of genus 3 and 4. However, there are no such flows on surfaces of genus 2. Thus, the following seems especially interesting.


\begin{prob}
Is there is a mixing, smooth, area preserving flow on a surface of genus 2 with finitely many fixed points, all non-degenerate, and no saddle connections?
\end{prob}
Another natural question is the size of the exceptional set we are considering.
\begin{prob} For fixed genus $g$, what is the Hausdorff dimension of the set of interval exchange transformations that are the first return map of a smooth area preserving flow on a surface of genus $g$ that is mixing, has no saddle connections, and has only finitely many fixed points, all of which are non-degenerate?
\end{prob}

\bold{Some pointers to the literature.} We have tabulated over 40 papers that should be mentioned in any complete history of area preserving flows on surfaces. We will omit many of them here. See \cite{koc_survey} for a survey. 

Some early examples of flows on surfaces that are not mixing include \cite{kol, kat67, Sthesis, koc_nomix, koc76}. The first result on weak mixing of flows may perhaps be due to von Neuman \cite{vN, FL09}. 

Novikov has suggested a link between area preserving flows on surfaces, and solid state physics \cite{Nov}. Because of this connection to physics, area preserving flows on surfaces are often called multi-valued Hamiltonian flows.

Arnold has pointed out that flows over interval exchange transformations with asymmetric logarithmic singularities arise from non-minimal flows on surfaces, and as a result these flows are now well studied and it is known that mixing in this context is generic \cite{ulc_asym}, see also \cite{Arn, koc_1, koc_2, koc_well, koc_some, SK}.  

Fr{{a}}czek-Lema{n}czyk have provided many examples of area preserving flows on surfaces that are disjoint from all mixing flows, and have proven weak mixing for many suspension flows over rotations \cite{FrLem, Fr2003}. 

Ko{{c}}hergin has given examples of  flows over rotations that mix at polynomial speed on rectangles \cite{koc_polymix}. Fayad-Kanigowski have shown multiple mixing of many suspension flows over rotations with asymmetric logarithmic singularities or degenerate fixed points \cite{FK}.
 
 Suspension flows over interval exchange transformations with different roof functions are also frequently studied, see for example \cite{LemNoMix, koclip, FL, FL06}.

\bold{The organization of this paper.} Most of this paper is occupied with the proof of Theorem \ref{T:genus2}. In the final Section \ref{S:surface}  we  explain how to modify our construction to yield Theorem \ref{T:main}. 

Section \ref{S:background} collects standard results on continued fractions and rotations whose use will be ubiquitous in our analysis. In Section \ref{S:alpha} we list the assumptions we must place on the rotation number of the base rotation, and we give an explicit example of a continued fraction expansion satisfying these assumptions. In Section \ref{S:skew} we define the skew product $T$ used to prove  Theorem \ref{T:genus2}, and discuss its non-minimal approximates $T_k$. In Sections \ref{S:Warmup} and \ref{sec:work} we prove estimates on Birkhoff sums. In Section \ref{sec:ue} we prove unique ergodicity of the $T$ used in Theorem \ref{T:genus2}, and in Section \ref{S:mix} we prove Theorem \ref{T:genus2}.

\bold{Big $O$ and little $o$ notation.} Given two sequences of numbers, $\{c_n\}$ and $\{d_n\}$, we write $c_n=O(d_n)$ if there exists $M\in \mathbb{R}$ so that $-M|d_n|<c_n<M|d_n|$ for all $n$. We write $c_n=o(d_n)$ if $\underset{n \to \infty}{\lim}\, \frac{c_n}{d_n}=0$.

\bold{Warning.} Readers should pay careful attention to the typesetting in subscripts, for example to distinguish $q_{n_{k+1}}$ and $q_{n_{k}+1}$.

\bold{Acknowledgements.} Research of JC partially supported by the NSF grants DMS 1004372, 1300550. Research of AW partially supported by a Clay Research Fellowship. The authors are very grateful to Anatole Katok, Corinna Ulcigrai, and Amie Wilkinson for useful  conversations.

\section{Continued fractions and rotations}\label{S:background}

\bold{Continued fractions.}
Fix a positive irrational real number $\alpha\in \bR$. Let $a_n$ denote the $n$-th term in the continued fraction expansion of $\alpha$, and let $p_n/q_n$ denote the $n$-th best approximate of $\alpha$. 
\begin{equation*}
\alpha=a_0+\cfrac{1}{a_1+\cfrac{1}{a_2+\cfrac{1}{a_3+\cdots}}}
\quad\quad \quad\quad
\frac{p_n}{q_n}=a_0+\cfrac{1}{a_1+\cfrac{1}{a_2+\cfrac{1}{\cdots+\cfrac{1}{a_n}}}}
\end{equation*}

A best approximate of  $\alpha$ is defined to be any rational number $p/q$ such that if $p',q'$ are integers with $0<q'\leq q$, then 
\begin{equation*} 
\left|\alpha-\frac{p}{q}\right|\leq \left|\alpha-\frac{p'}{q'}\right|.
\end{equation*}

\begin{thm}\label{T:CF}
The following hold. 
\begin{enumerate}
\item Recursive formulas: 
\begin{eqnarray*}
 p_{n+1}&=&a_{n+1}p_n+p_{n-1},\\
 q_{n+1}&=&a_{n+1}q_n+q_{n-1},\\
 \frac{p_n}{q_n}&=&\frac{p_{n-1}}{q_{n-1}}+\frac{(-1)^{n+1}}{q_{n-1}q_n}.
 \end{eqnarray*} 
\item Alternating property: 
\begin{equation*}
\frac{p_0}{q_0}<\frac{p_2}{q_2}< \cdots<\alpha< \cdots< \frac{p_3}{q_3}<\frac{p_1}{q_1}.
\end{equation*}
\item Upper and lower bounds: 
\begin{equation*}
\frac 1 {q_k(q_k+q_{k+1})}<|\alpha-\frac{p_k}{q_k}|<\frac 1 {q_kq_{k+1}}.
\end{equation*}
\item Best approximates property: The set of best approximates of $\alpha$ is exactly equal to $\{\frac{p_n}{q_n}\}_{n=1}^\infty$.
\end{enumerate}
\end{thm}

For proofs of any of these facts, see \cite{khinchin}. In particular, see page 36 for the upper and lower bounds.

\bold{Rotations.} Let $S^1=\bR/\bZ$ denote the circle, and let $R:S^1\to S^1$ denote rotation by $\alpha$, so $R(x)=x+\alpha.$ We will often implicitly identify the circle with the interval $[0,1)$. 

Let $d$ denote the distance on the circle coming from the standard distance on $\bR$. 

Given $x\in S^1$, define a closest return time as positive integer $q$ such that if $0<q'<q$, then $$d(x,R^{q'}(x))>d(x,R^q(x)).$$ This definition does not depend on $x$. Define the orbit segment of length $q$ of $x\in S^1$ to be the sequence $\{R^i(x)\}_{i=0}^{q-1}$.  

We will say that a subset of $S^1$ is \emph{$\delta$-separated} if the distance between any two distinct points in the subset is at least $\delta$. 

\begin{thm}\label{T:R}
The following hold. 
\begin{enumerate}
\item Alternating property: For $n$ odd $R^{q_n}(x)\in x+(-\frac12,0)$, and for $n$ even $R^{q_n}(x)\in x+(0,\frac12)$.
\item Upper and lower bounds: 
\begin{equation*}
\frac 1 {q_n+q_{n+1}}< d(x, R^{q_n}x) < \frac 1 {q_{n+1}}.
\end{equation*}
\item Best approximates property: The closest return times are exactly $\{q_n\}_{n=1}^\infty$.
\item Separation: Any orbit segment of length at most $q_n$ is at least $d(x, R^{q_{n-1}}x)$ separated. 
\item Equidistribution: $\{R^i0\}_{i=0}^{q_n-1}$ contains exactly one point of each interval $[\frac{i}{q_n}, \frac{i+1}{q_n})$, for $i=0, \ldots, q_n-1$. 
\item Denjoy-Koksma Inequality: For any function $g:[0,1)\to\bR$ of bounded variation, and any $x\in S^1$,
$$\left| \sum_{i=0}^{q_n-1} g(R^i(x)) - q_n \int g\right| \leq \Var(g),$$
where $\Var(g)$ is the total variation of $g$. 
\end{enumerate}
\end{thm}

The first three statements follow from the corresponding statements in Theorem \ref{T:CF}. Separation follows from the best approximates property. The equidistribution property follows in an elementary way from the bounds in Theorem \ref{T:CF}, and the Denjoy-Koksma inequality follows from the equidistribution property. For proofs of the equidistribution property and Denjoy-Koksma, we highly recommend the blog post of Lima \cite{Mat} (see Lemma 5 and Theorem 6), which is partially based on the paper by Herman \cite{Her} where Denjoy-Koksma was first proven. 

\section{Picking $\alpha$}\label{S:alpha}

In this paper, we will require  $\alpha$ with very special properties. Precisely, we will require the existence of a subsequence $n_k$ of the positive integers such that the following assumptions on $n_k$ and the continued fraction expansion of $\alpha$ hold.  Let $\fp{y}$ denote the fractional part of a real number $y$, so $\fp{y}\in[0,1)$. Note that by the Alternating Property, $d(x, R^{q_i}(x))=\fp{q_i \alpha}$ if $i$ is even, and  $d(x, R^{q_i}(x))=1-\fp{q_i \alpha}$ if $i$ is odd.

\bold{Assumptions.}
\begin{enumerate}
\item \label{A:even}
All $n_k$ are even.
\item \label{A:12}
$\sum_{k=1}^\infty 2\fp{q_{n_k} \alpha}<\min(\frac{\alpha}4,1-\frac{\alpha}4). $
\item \label{A:toavoid0}
For all $\ell\geq 1$, we have $1+\sum_{k=1}^\ell 2 q_{n_k} < q_{n_\ell+1}.$
\item \label{A:growing}
$\underset{\ell \to \infty}{\lim}\, a_{n_{\ell}+1}=\infty$.
\item \label{A:a_nsmall}
$a_{n}=o(\log \log(q_n))$, and $a_{n+1}=o(\log \log(q_n))$. 
\item \label{A:infsum} 
$\sum_{k=1}^\infty a_{n_k+1}^{-1}=\infty.$
\item \label{A:a_nsmallsum}
$\sum_{i=1}^{n+1} a_i = O(\log(q_n))$. 
\item \label{A:rapid up}
$\log(q_{n_{k-1}+1})=o(\log(q_{n_k}))$.
\item \label{A:2}
$a_{n_k}=2$ and $a_{n_k-1}=2$. 
\end{enumerate}
These assumptions are in effect for the remainder of this paper. They are stronger than necessary; we saw no benefit in trying to pick the weakest sufficient assumptions. 

\begin{rem}
Assumptions \ref{A:even}, \ref{A:12}, and \ref{A:toavoid0} allow us to fix notation. They have no particular dynamical significance for us. Assumption \ref{A:toavoid0} is used in proving unique ergodicity in Section \ref{sec:ue}, and a few other places, where it guarantees that the sequence $q_{n_k}$ grows very rapidly. 

Assumption \ref{A:growing} indicates that at time $q_{n_{\ell}}$, orbits come back very close to themselves, and for a long time after that they almost repeat their paths. This will govern transfer of mass between invariant subsets of certain non-minimal approximates of $T$. Assumption \ref{A:infsum} guarantees that $T$ is uniquely ergodic.

Assumptions \ref{A:a_nsmall} and \ref{A:a_nsmallsum} reflect that the continued fraction has mostly small coefficients, and hence $\alpha$ is poorly approximated by rationals, so the rotation by $\alpha$ has especially good equidistribution properties. These assumptions are important in estimating Birkhoff sums.

Assumption \ref{A:rapid up} indicates that the times $n_k$ are chosen so sparsely that orbits of length $q_{n_{k}}$ come back vastly closer to themselves than orbits of length $q_{n_{k-1}}$. It will be used in Lemma \ref{L:insignificant1}. 

Assumptions \ref{A:2} is  used  in the proofs of unique ergodicity of $T$ and $\hat{T}$, the multi-interval exchange  that we use to prove Theorem \ref{T:main}.
\end{rem}

 To verify that the assumptions are mutually compatible, we show the following. 
\begin{prop}
Define a sequence $n_k$ recursively by $$n_1=10\quad\quad\text{and}\quad\quad n_k=10^{k^2}n_{k-1}.$$ Define $\alpha$ by specifying its continued fraction expansion as follows: 
$$a_{n_{k}+1}=k+8 \text{ for all }k,\quad\quad \text{and}\quad\quad a_i=2\text{ if }i\notin\{n_k+1\}_{k\in \mathbb{N}}.$$ Then all the above assumptions are satisfied. 
\end{prop}
\begin{proof}
We will establish the assumptions one at a time. 
\begin{enumerate}
\item Obvious. 
\item By the upper bound in Theorem \ref{T:CF}, $\fp{q_{n_k} \alpha}< \frac 1 {q_{n_k+1}}$. By the second recursive formula, and the fact that all $a_i$ are at least 2, we see that $q_{n}>2^n$. Since $n_1=10$, $\sum_{k=1}^\infty 2\fp{q_{n_k} \alpha}<\frac{4}{2^{10}}$. The alternating property in Theorem \ref{T:CF} (using $p_1/q_1$ and $p_2/q_2$) gives that $\frac25< \alpha <\frac 12$, so the result follows. 
\item Again by the recursive formula and the fact that all $a_i$ are at least 2, we have $q_{i+k}> 2^k q_i$. Since $a_{n_{k+1}}>2$
 for every $k$, we have in particular that $q_{n_k+1}> 3q_{n_k}$, whence the inequality follows by induction. 
\item Obvious.
\item Clear if $i \neq n_{k}, n_{k}+1$, and otherwise it follows  because $k^2=o(\log(n_k))$ and $q_i>2^i$. 
\item Obvious.
\item Let $i$ be the greatest number such that $n_i\leq n$. Then there are $i=a_{n_i+1}-8$ numbers $k$ such that $n_k\leq n$.
Thus by \ref{A:a_nsmall}, there are  $o(\log\log q_n)$ numbers $k$ such that $n_k\leq n$. Again by 5, each $a_{n_k+1}$ is $o(\log\log q_n)$. So, the sum is at most $2n + o(\log\log q_n)^2$. Since $n$ is $O(\log q_n)$, we have the claim. 

\item Since $q_{i+k}> 2^k q_i$, we have that $q_{n_{k+1}}> 2^{n_{k+1}-n_k} q_{n_k}$, so 
$$\log q_{n_{k+1}}> (n_{k+1}-n_k)\log(2) + \log q_{n_k}.$$
On the other hand, $q_{i+1}\leq (a_{i+1}+1)q_i$, so we have that 
$$\log(q_n)\leq \sum_{i=1}^n \log(a_{i}+1)=O(n),$$
since there are $O(n)$ terms where $\log(a_{i}+1)$ has size $\log(3)$, and at most $O(\log(n))$ terms where $i=n_k+1$ for some $k$ and hence  $(a_{i+1}+1)$ has size $\log(k+9)=O(\log(n))$.

Since $(n_{k+1}-n_k)/n_k\to \infty$, the result follows. 
\item Obvious.
\end{enumerate}
\end{proof}

\section{A skew product over a rotation}\label{S:skew}
This section develops the Katok-Sataev-Veech  construction \cite{v_skew, kat_skew, sat_skew} in a manner convenient for our purposes.

We now define a $\bZ_2=\{0,1\}$ skew product over the circle rotation $R$ by angle $\alpha$. Set $$J=\left[0, \sum_{k=1}^\infty 2\fp{q_{n_k} \alpha}\right).$$
By Assumption \ref{A:12}, we have that $J$ has length less than $1$. 

Define a skew product $T$ to be the rotation by $\alpha$ skewed over the interval $J$, so
$$T:S^1\times\bZ_2\to S^1\times\bZ_2, \quad \quad T(x, j) = (R(x), j+\chi_J(R(x))).$$
Here $\chi_J$ denotes the characteristic function of $J$. We will denote the coordinate projections as 
$$\pi_{S^1}:S^1\times\bZ_2\to S^1,\quad\quad\pi_{\bZ_2}:S^1\times\bZ_2\to\bZ_2.$$
We denote by $\iota:S^1\times\bZ_2\to S^1\times\bZ_2$ the involution $\iota(x,j)=(x,j+1)$.

\textbf{Notation.} 
Set $$J_\ell=\left[\sum_{k=1}^{\ell-1} 2\fp{q_{n_k} \alpha},\sum_{k=1}^{\ell} 2\fp{q_{n_k} \alpha}\right),$$
so  $J$ is the disjoint union of the $J_\ell$. Let 
\begin{eqnarray*}
J_\ell'&=&\left[\sum_{k=1}^{\ell-1} 2\fp{q_{n_k} \alpha},\sum_{k=1}^{\ell-1}2\fp{q_{n_k} \alpha}+\fp{q_{n_\ell} \alpha}\right),\\
J_\ell''&=&\left[\sum_{k=1}^{\ell-1} 2\fp{q_{n_k} \alpha}+\fp{q_{n_\ell} \alpha},\sum_{k=1}^{\ell}2\fp{q_{n_k} \alpha}\right)
\end{eqnarray*}
be the left and right halves of the interval $J_\ell$, so $J_\ell=J_\ell'\cup J_\ell''.$ 
 
\begin{figure}[h]
\includegraphics[scale=0.4]{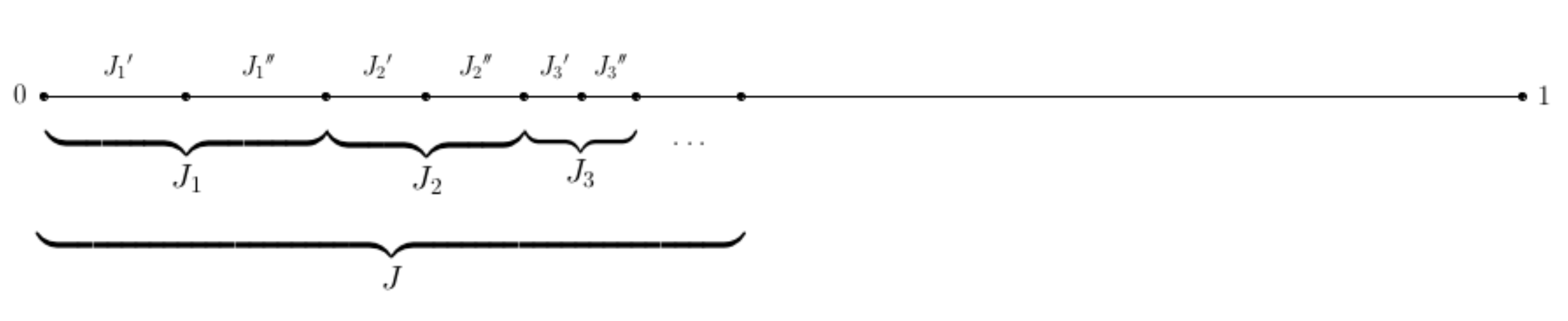}
\end{figure}

For any subset $Y\subset S^1$, we set $$\tilde{Y}=Y\times \bZ^2=\pi_{S^1}^{-1}(Y).$$

\textbf{Moving a bit of one invariant set into another.} Here we consider a very general construction. The notation has been chosen to match the situation to which the lemma will be applied.  
\begin{lem}\label{L:UV} 
Let $T_\ell:X\to X$ be an invertible transformation with two invariant sets $U_\ell$ and $V_\ell=X\setminus U_\ell$. Suppose there is an involution $\iota:X\to X$ that interchanges $U_\ell$ and $V_\ell$. Let $\tilde{J}'_{\ell+1}\subset X$ be $\iota$ invariant, and let $q_{n_{\ell+1}}$ be any integer such that $$\tilde{J}'_{\ell+1}, T_\ell(\tilde{J}'_{\ell+1}), \ldots, T_\ell^{q_{n_{\ell+1}}}(\tilde{J}'_{\ell+1})$$ are disjoint. Consider the transformation $T_{\ell+1}$ on $X$ defined by 
$$T_{\ell+1}(x) = \begin{cases}
\iota(T_\ell(x)) &\mbox{if } T_\ell(x)\in \tilde{J}'_{\ell+1}\cup T_\ell^{q_{n_{\ell+1}}}(\tilde{J}'_{\ell+1}) \\
T_\ell(x) &\mbox{otherwise.}
\end{cases}$$
Then the set 
$$U_{\ell+1}\quad=\quad \left(U_\ell\quad\cup\quad\bigcup_{i=0}^{q_{n_{\ell+1}}-1}T_\ell^i(\tilde{J}'_{\ell+1}\cap V_\ell)\right)\quad\setminus\quad \bigcup_{i=0}^{q_{n_{\ell+1}}-1}T_\ell^i(\tilde{J}'_{\ell+1} \cap U_\ell)$$ 

is $T_{\ell+1}$ invariant, as is $V_{\ell+1}=X\setminus U_{\ell+1} = \iota(U_{\ell+1}).$
\end{lem}

\begin{figure}[h]
\includegraphics[scale=0.5]{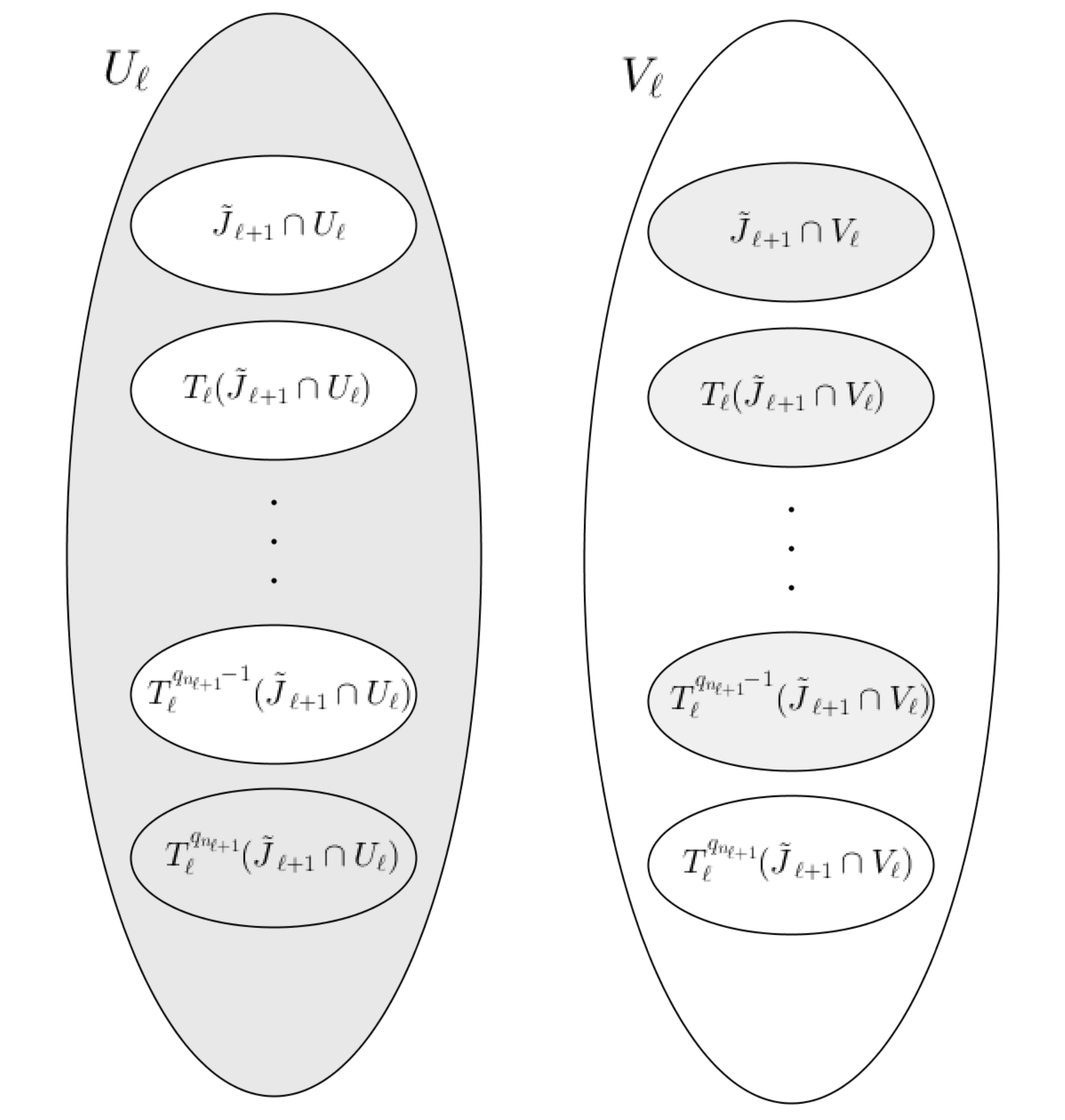}
\caption{Schematic picture of Lemma \ref{L:UV}. The set $U_{\ell+1}$ is shaded.} 
\label{F:UV}
\end{figure}

The proof of Lemma \ref{L:UV} is left to the reader, who is invited to convince himself or herself while looking at Figure \ref{F:UV}.

\textbf{Non-minimal approximates.}
Define $T_\ell:[0,1]\times\bZ_2\to[0,1]\times\bZ_2$ to be the rotation by $\alpha$ skewed over the interval $[0, \sum_{k=1}^\ell 2\fp{q_{n_k} \alpha})$, so
$$T_\ell(x, j) = (R(x), j+\chi_{[0, \sum_{k=1}^\ell 2\fp{q_{n_k} \alpha})}(R(x))).$$

Set $U_0=[0,1]\times\{1\}$ and $V_0=[0,1]\times\{0\}$, and for each $\ell\geq 0$ define
\begin{eqnarray*}
U_{\ell+1}\quad= &&\left(U_\ell\quad\cup\quad\bigcup_{i=0}^{q_{n_{\ell+1}}-1}T_\ell^i(\tilde{J}'_{\ell+1}\cap V_\ell)\right)
\\&\setminus& \bigcup_{i=0}^{q_{n_{\ell+1}}-1}T_\ell^i(\tilde{J}'_{\ell+1}\cap U_\ell)
\end{eqnarray*}
and $V_{\ell+1}=U_{\ell+1}^c$. 

We will show that $U_\ell$ and $V_\ell$ are $T_\ell$ invariant sets. Because $T_\ell$ is not minimal, and is equal to $T$ on a set of large measure, we refer to $T_\ell$ as a non-minimal approximate to $T$. Our understanding of $T$ will follow from a study of the $T_\ell$. Our understanding of $T_\ell$ will be inductive, using the fact that $T_\ell$ is almost the same as $T_{\ell-1}$ but with additional skewing.

\begin{lem}\label{L:dis}
For each $\ell>1$ the intervals
$$R^i(J'_\ell), \quad i=0,\ldots, q_{n_\ell}$$
are disjoint from each other and $[0,\fp{q_{n_{\ell}}\alpha})$ and $[1-\fp{q_{n_{\ell}}\alpha}, 1)$.
\end{lem}
Recall $J_\ell'=[\sum_{k=1}^{\ell-1} 2\fp{q_{n_k} \alpha},\sum_{k=1}^{\ell-1}2\fp{q_{n_k} \alpha}+\fp{q_{n_\ell} \alpha}).$

\begin{proof}
Since $J'_\ell= R^{N_0}([0,\fp{q_{n_{\ell}}\alpha}))$ when $N_0=2\sum_{k=1}^{\ell-1} q_{n_k}$, the intervals in question are
$$R^{N}([0,\fp{q_{n_{\ell}}\alpha}), \quad N=-q_{n_{\ell}} \text{  and  } N=0 \text{  and  }N=N_0, \ldots, N_0+q_{n_\ell}.$$

Thus the intervals in question are contained in the $N_0+2q_{n_\ell}+1$ orbit of an interval of length  $\fp{q_{n_{\ell}}\alpha}$. By Assumption \ref{A:toavoid0}, $N_0+2q_{n_\ell}+1<q_{n_\ell+1}$, so the separation property gives that these intervals are disjoint. 
\end{proof}

\begin{lem}\label{L:invar}
$U_\ell$ and $V_\ell$ are $T_\ell$ invariant, and $V_\ell=\iota (U_\ell).$
\end{lem}

\begin{proof}
This is proven inductively using Lemma \ref{L:UV}, where the notation has been chosen to indicate how the result should now be applied.

$T_{\ell+1}$ is obtained from $T_\ell$ by additionally skewing over $J_{\ell+1}=J_{\ell+1}'\cup R^{q_{n_{\ell+1}}}(J_{\ell+1}').$ This additional skewing amounts to applying $\iota$ every time the orbit lands in $\tilde{J}_{\ell+1}'\cup T_\ell^{q_{n_{\ell+1}}}(\tilde{J}_{\ell+1}').$ Thus the definition of $T_{\ell+1}$ as a skew product coincides with the inductive definition of $T_{\ell+1}$ provided in Lemma \ref{L:UV}.

The disjointness condition in Lemma \ref{L:UV} has been verified in the previous lemma.
\end{proof}

\begin{thm}\label{T:skew}
For each integer $\ell\geq1$, there are sets $U_\ell, V_\ell$ whose disjoint union is $S^1\times\bZ_2$ such that $\iota(U_\ell)=V_\ell$ and the following properties hold. 
\begin{enumerate}
\item $U_{\ell}$ contains
  $$[0, \fp{q_{n_{\ell}}\alpha})\times \{0\}\quad\text{and}\quad[1-\fp{q_{n_{\ell}}\alpha},1)\times \{1\}.$$
\item There is a subinterval $J_\ell'$ of $S^1$ of length $\fp{q_{n_\ell}\alpha}$ such that 
$$\pi_{S_1}(U_\ell\setminus U_{\ell-1})=\cup_{i=0}^{q_{n_\ell-1}} R^i(J_\ell').$$
\end{enumerate}
\end{thm}

\begin{rem}
In can be shown (using $(1)$ and $(2)$  and Theorem \ref{T:R}) that $U_\ell$ contains most of $[0, \fp{q_{n_{\ell-1}}\alpha})$. This is morally important in the proof of Claim (3) of Theorem \ref{thm:output jon}.  
\end{rem}

\begin{proof}[{Proof of Theorem \ref{T:skew}.}]
The first part of the first claim is true by induction. The base case is $\ell=1$. The inductive step follows from the  definition of $U_{\ell+1}$ as well as the disjointness of $\{R^i(J'_\ell)\}_{i=0}^{q_{n_\ell}-1}$ from $[0,\fp{q_{n_{\ell}}\alpha})$ and $[1-\fp{q_{n_{\ell}}\alpha}, 1)$ in Lemma \ref{L:dis}. 

The second claim is by the definition of $U_\ell$.
\end{proof}

\section{Birkhoff sums for the rotation $R$}\label{S:Warmup}

In this section we consider the Birkhoff sums of the function $g(x)=\frac 1 x$ over the rotation $R$. This is used in the next section to provide the shearing estimates.

Results on Birkhoff sums over rotations  have previously been used to prove mixing in a slightly different setting (asymmetric singularities), see for example \cite{koc_1,SK}.

\begin{lem}
For each positive integer $N$, there are unique integers $b_n$ such that 
$$N=\sum_{n=1}^k b_n q_n,$$
 and such that $0\leq b_n\leq a_{n+1}$ and $q_n>\sum_{i=0}^{n-1}b_iq_i$ for each $n$.  
\end{lem}
Such an expression is called  the Ostrowski expansion of $N$ \cite{Ost}.
\begin{proof} 
 Pick $k$ such that $q_{k+1}>N\geq q_k$, and let $b_k$ be the unique integer such that $0\leq N-b_kq_k< q_k$. By the second recursive formula in Theorem \ref{T:CF}, 
$$N-a_{k+1} q_k < q_{k+1}-a_{k+1}q_k= q_{k-1}.$$
So we get that $b_k\leq a_{k+1}.$ Replacing $N$ with 
$N-b_kq_k$ and iterating this procedure gives the Ostrowski expansion for $N$. 
\end{proof}

\begin{prop}
Let $g:(1,0]\to [0,\infty)$ be a monotone decreasing function. Let 
$$g_N(x)=\max_{i=0, \ldots, N-1} g(R^i(x)).$$
Let $q_{k+1}>N\geq q_k$, and let $1\leq C_N< 2q_{k+1}$. Then, without any assumptions on $\alpha$,
\begin{eqnarray*} 
\sum_{i=0}^{N-1} g(R^ix) &=& g_N(x) + N \int_\frac{C_N}{2q_{k+1}}^1 g
\\&+& O\left( g\left(\frac{C_N}{2q_{k+1}}\right)\sum_{i=2}^{k+1} a_i +\sum_{i=1}^{\lfloor C_N \rfloor} g\left(\frac{i}{2q_{k+1}}\right) \right).
\end{eqnarray*}
\end{prop}

\begin{proof}
Since $N<q_{k+1}$, each orbit segment of length $N$ is $\fp{q_{k}\alpha}\geq \frac1{2q_{k+1}}$ separated. Thus, 
$$-g_N(x) + \sum_{i=0}^{N-1} g(R^ix) $$ 
is within $O\left(g\left(\frac1{2q_{k+1}}\right)\right)$ of 
$$\sum_{i=0}^{N-1} g(R^ix)\chi_{\left[ \frac1{2q_{k+1}}, 1\right)}(R^i x).$$
Indeed, if $g_N(x)>2q_{k+1}$ they are equal and otherwise the claim is immediate by monotonicity.

Now, note that by separation and monotonicity, 
\begin{eqnarray*}
\sum_{i=0}^{N-1} g(R^ix)\chi_{\left[ \frac1{2q_{k+1}}, \frac{C_N}{2q_{k+1}}\right)}(R^i x) 
&\leq& \sum_{i=1}^{\lfloor C_N \rfloor} g\left(\frac{i}{2q_{k+1}}\right).
\end{eqnarray*}
Now, consider the Ostrowski expansion $N=\sum_{n=1}^k b_n q_n$, and consider an orbit segment of length $N$ to be a suite of $b_n$ orbit segments of length $q_n$.  Applying Denjoy-Koksma individually to these $\sum_{n=1}^k b_n$ orbit segments, we get 
$$\sum_{i=0}^{N-1} g(R^ix)\chi_{\left[\frac{C_N}{2q_{k+1}}, 1\right)}(R^i x) =N\int_\frac{C_N}{2q_{k+1}}^1 g + O\left( g\left(\frac{C_N}{2q_{k+1}}\right)\sum_{i=2}^{k+1} a_i\right).$$
Here we have used the estimate $\sum_{n=1}^k b_n\leq \sum_{i=2}^{k+1} a_i$, which is immediate from the definition of the Ostrowski expansion.
\end{proof}

\begin{cor}\label{C:sums}
With the assumptions in this paper on $\alpha$, if $g(x)=1/x$, then for all $N$
\begin{eqnarray*} 
\sum_{i=0}^{N-1} g(R^ix) &=& g_N(x) + N \log(N) +o(N\log(\log(N))^2)
\end{eqnarray*}
and 
\begin{eqnarray*} 
\sum_{i=0}^{N-1} g'(R^ix) &=& g'_N(x) + o\left( \left(N \log\log(N)\right)^2\right).
\end{eqnarray*}
\end{cor}

\begin{proof} Choose $k$ such that $q_{k}\leq N<q_{k+1}$. Let $0< C_N\leq q_{k+1}$ be a constant, and let us address the different quantities appearing in the proposition above. 

Recall Assumption \ref{A:a_nsmall}, which gives that $q_{k+1}/N=o(\log \log(N))$, and note
\begin{eqnarray*}
\int_\frac{C_N}{2q_{k+1}}^1 g &=& \log(N)+\log(q_{k+1}/N)+\log(2)-\log(C_N)
\\&=&\log(N)-\log(C_N)+\log(o(\log \log(N))).
\end{eqnarray*}

Next, note that the same bound for $q_{k+1}/N$ gives  
\begin{eqnarray*}
\sum_{i=1}^{\lfloor C_N \rfloor} g\left(\frac{i}{2q_{k+1}}\right) &=& 2q_{k+1} \sum_{i=1}^{\lfloor C_N \rfloor} \frac1i
\\&=& 2q_{k+1} \log(C_N) + O(q_{k+1})
\\&=& o(\log \log(N) N \log C_N).
\end{eqnarray*}
 
Finally,  note that Assumption \ref{A:a_nsmallsum} gives that $\sum_{i=2}^{k+1} a_i=O(\log(N))$,
and hence 
\begin{eqnarray*}
g\left(\frac{C_N}{2q_{k+1}}\right)\sum_{i=2}^{k+1} a_i  &=& O\left( \frac{2q_{k+1}}{C_N}  \log(N)\right)
\\&=& o\left( \frac{N \log \log(N)}{C_N}  \log(N)\right)
.\end{eqnarray*}

Now setting $C_N=\log(N)$ and using the previous proposition gives the result. 

The second bound is similar. 
\end{proof}

\begin{lem}\label{L:insignificant1}
Set $g(x)=1/x$. 
For all $N\geq q_{n_k}$,
$$ \sum_{i=0}^{N-1} g(R^ix)\chi_{[\fp{q_{n_{k-1}}\alpha},1)}(R^ix) = o(N\log(N)) $$
and 
$$ \sum_{i=0}^{N-1} g'(R^ix)\chi_{[\fp{q_{n_{k-1}}\alpha},1)}(R^ix) = O(N^2).$$
\end{lem}

\begin{proof}
 Let $N=\sum_{n=1}^k b_n q_n$ be the Ostrowski expansion of $N$. As in the previous proof, using Assumption \ref{A:a_nsmallsum} we get $\sum_{n=1}^k b_n =O(\log(N))$. 
Using  Denjoy-Koksma, we get 
\begin{eqnarray*}
\sum_{i=0}^{N-1} g(R^ix)\chi_{[\fp{q_{n_{k-1}}\alpha},1)}(R^ix) 
&=&
O\left(-N \log\fp{q_{n_{k-1}}\alpha} + \frac{\log(N)}{\fp{q_{n_{k-1}}\alpha}}\right)
\\&=&
O\left(N \log(q_{n_{k-1}+1}) + \log(N)q_{n_{k-1}+1}\right).
\end{eqnarray*}

Assumption \ref{A:rapid up}, which gives $\log(q_{n_{k-1}+1})=o(\log(N))$, gives the result. 
Similarly, to prove the second estimate it suffices to note
\begin{eqnarray*}
\sum_{i=0}^{N-1} g'(R^ix)\chi_{[\fp{q_{n_{k-1}}\alpha},1)}(R^ix) 
&=&
O\left(\frac{N}{\fp{q_{n_{k-1}}\alpha}} + \frac{\log(N)}{\fp{q_{n_{k-1}}\alpha}^2}\right)
\\&=&
O\left(N q_{n_{k-1}+1} + \log(N)q_{n_{k-1}+1}^2\right) \\&=& O(N^2).
\end{eqnarray*}
\end{proof}

\begin{lem}\label{L:insignificant2}
Set $g(x)=1/x$. Let $S$ be an orbit of length $q_{n_\ell}$ of an interval of length $\fp{q_{n_{\ell}}\alpha}$. Assume $S$ is disjoint from $[0, \fp{q_{n_{\ell}}\alpha})$. Let $x$ be any point disjoint from $\cup_{i=0}^{\sqrt{a_{n_{\ell}+1}} q_{n_{\ell}}}R^{-i}(S)$. Then for all $N>q_{n_\ell}$
$$ \sum_{i=0}^{N-1} g(R^ix)\chi_{S}(R^ix) =
o(N\log(N))$$
and 
$$ \sum_{i=0}^{N-1} g'(R^ix)\chi_{S}(R^ix) =
o(N^2\log(N)^\frac13).$$
The implied constant does not depend on $\ell$. 
\end{lem} 


\begin{proof}
Denjoy-Koksma gives that the sum of 
$$g(x)\chi_{[\fp{q_{n_{\ell}}\alpha},1)}(x)$$
over an orbit of length  $q_{n_{\ell}}$ is at most
$$O(q_{n_{\ell}}\log q_{n_{\ell}+1} +  q_{n_{\ell}+1})=O(q_{n_{\ell}}\log q_{n_{\ell}+1} )=O(q_{n_{\ell}}\log q_{n_{\ell}} ).$$
Note that, by the separation property, each orbit of length $q_{n_\ell+1}> a_{n_\ell+1}q_{n_\ell}$ can make at most one pass through $S$. A point $x$ as in the lemma stays outside of $S$ for time at least $\sqrt{a_{n_{\ell}+1}} q_{n_{\ell}}$, then makes a pass through $S$, then stays outside of $S$ for time at least (actually much more than) $\sqrt{a_{n_{\ell}+1}} q_{n_{\ell}}$, then makes a pass through $S$, etc. Therefore, if the orbit makes $m-1$ full passes through $S$, plus possibly a final partial pass through, then the first sum in the lemma statement is at most 
$$mO(q_{n_{\ell}}\log q_{n_{\ell}} ),$$
while $N$ is at least 
$$ m \sqrt{a_{n_{\ell}+1}} q_{n_{\ell}}.$$
So $N \log N$ is at least 
$$m \sqrt{a_{n_{\ell}+1}} q_{n_{\ell}} \log (q_{n_{\ell}}),$$
whence the result follows by Assumption \ref{A:growing}. 

Similarly, Denjoy-Koksma gives that the sum of 
$$g'(x)\chi_{[\fp{q_{n_{\ell}}\alpha},1)}(x)$$
over an orbit of length  $q_{n_{\ell}}$ is at most
$$O(q_{n_{\ell}} q_{n_{\ell}+1} +  q_{n_{\ell}+1}^2)=O(a_{n_{\ell}+1}^2 q_{n_{\ell}}^2).$$ As before, if the orbit makes $m-1$ full passes through $S$, plus possibly a final partial pass though, then the sum is at most 
$$mO(a_{n_{\ell}+1}^2 q_{n_{\ell}}^2),$$
while $N$ is at least 
$$ m \sqrt{a_{n_{\ell}+1}} q_{n_{\ell}}.$$
So $N^2 (\log N)^\frac13$ is at least 
$$m^2 {a_{n_{\ell}+1}} q_{n_{\ell}}^2 \log (q_{n_{\ell}})^\frac13,$$
whence the result follows by Assumption \ref{A:a_nsmall}.
\end{proof}

\section{Birkhoff sums for the skew product $T$}\label{sec:work} 

We now define a function $f:S^1\times \bZ_2\to \bR$, which will serve as the roof function for a suspension flow over $T$. Recall that $d$ denotes distance on the circle $S^1$. Also note that $$T(x, j) = (R(x), j+\chi_J(R(x)))$$ has discontinuities at  $x=R^{-1}(0)$ and $x=R^{-1}(|J|)$. We will define $f$ to have logarithmic singularities over these discontinuities, as well as an additional logarithmic singularity over $0$. 

\begin{eqnarray*}
f(x,j)
=
1&+&|\log(d(x,R^{-1}(0)))|
\\&+&|\log(d(x,R^{-1}(|J|)))|
\\&+&
\chi_{\{1\}}(j) \cdot |\log(d(x,0))|
\end{eqnarray*}

Let $\lambda$ denote Lebesgue probably measure on $S^1\times \bZ_2$.

\begin{thm} \label{thm:output jon}
For all large enough $M$ there exists $G(M)\subset S^1\times \bZ_2$ such that the following hold.
\begin{enumerate}
\item $G(M)$ is the disjoint union of intervals of length at least $\frac{2}{M\sqrt{\log(M)}}$.
\item $\lambda (G(M))\to 1$. 
\item If $N \in [\frac{M}2,2M]$ and $p\in G(M)$ then
$$\sum_{i=0}^{N-1}f'(T^ip)=\pm N\log(N)+o(N\log N).$$
\item If $N \in [\frac{M}2,2M]$ and $p\in G(M)$ then
$$\left|\sum_{i=0}^{N-1}f''(T^ip)\right|=o((\log(N))^{\frac 1 3}N^2).$$
\item If  $k<2M$ and $p\in G(M)$ then
$$\sum_{i=k}^{k+5\sqrt{\log(M)}}\left|f'(T^ip)\right|=o\left( {M\sqrt{\log(M)}}\right).$$
\item $T^i$ is continuous on each interval of $G(M)$ for all $0\leq i\leq 2M+5\sqrt{\log(M)}$.
\end{enumerate}
\end{thm}
The most important conditions are the first three. In the third, it is implicit that the sign is constant on each interval in $G(M)$.

\begin{proof}
Pick $\ell$ such that $q_{n_\ell}< \frac{M}2<q_{n_{\ell+1}}$. Let $Q_M\subset S^1$ be the points within distance $\frac{1}{M (\log M)^{\frac1{12}}}$ of the singularities of $f$. 
Define the ``bad set" in $S^1$ to be 
\begin{eqnarray*}
B_{S^1}(M)&=&\bigcup_{i=0}^{2M+5\sqrt{\log(M)}} R^{-i}(Q_M)
\\&\cup& \bigcup_{i=0}^{2M+5\sqrt{\log(M)}} R^{-i}(\cup_{j=\ell+1}^\infty J_j)
\\&\cup& \bigcup_{i=-q_{n_\ell}}^{\sqrt{a_{n_\ell+1}}q_{n_\ell}}R^{-i} (J_\ell').
\end{eqnarray*}
Set $B(M)=B_{S^1}(M)\times \bZ_2$. The complement of the bad set ${B}(M)^c$ is a union of disjoint intervals. We define the good set $G(M)$ to be the union of all those intervals in ${B}(M)^c$ of length at least $\frac{2}{M\sqrt{\log(M)}}$.

\begin{rem}
The three terms in the definition of $B_{S^1}(M)$ reflect the three things that can go wrong in our arguments. The first  is that an orbit can come too close to a singularity. 
 The second is that an orbit can  fail to stay in one of $U_{\ell+1}$ or $V_{\ell+1}$. The third is that an orbit can switch too soon between $U_\ell$ and $V_\ell$. Switching between $U_\ell$ and $V_\ell$ is undesirable but unavoidable, and this last term mitigates this problem.
\end{rem}

\bold{Claims (1) and (2):} Recall that $ \cup_{j=\ell+1}^\infty J_j $  is an interval of size 
$$\sum_{k=\ell+1}^\infty 2\fp{q_{n_k} \alpha}\leq \sum_{k=\ell+1}^\infty \frac{2}{q_{n_k+1}}\leq \frac{4}{q_{n_{\ell+1}+1}}$$
by the exponential growth of the $q_{n_k}$, for example by Assumption \ref{A:toavoid0}. Since $\frac{M}2<q_{n_{\ell+1}}$, Assumption \ref{A:growing} gives that $\frac{4}{q_{n_{\ell+1}+1}} = o(1/M)$.  
Hence the measure of the second union in $B_{S^1}(M)$ goes to zero. 

The third union has size at most three times 
$$ \frac{\sqrt{a_{n_{\ell}+1}} q_{n_{\ell}}}{q_{n_{\ell}+1}}= \frac{\sqrt{a_{n_{\ell}+1}} q_{n_{\ell}}}{a_{n_{\ell}+1}q_{n_{\ell}}+q_{n_{\ell}-1}}\to 0$$
by Assumption \ref{A:growing}. 

The first union obviously has size $o(M)$, so in total we see that $B(M)^c$ has measure $1-o(1/M)$. Note also that $B(M)$ is the disjoint union of $O(M)$ intervals. It remains only to show that the subset covered by intervals of length at least $\frac{2}{M\sqrt{\log(M)}}$ has measure going to 1. This is true since the complement has measure at most  $O\left(\frac{M}{M\sqrt{\log(M)}}\right)=o(1)$ (the total number of such intervals, times max length).   

\bold{Claim (3):} Let $g(x)=1/\fp{x}$, so for example $g(-0.1)=\frac {10} 9$. The difference between $f'(x,j)$ and
\begin{eqnarray*}
&&g(-x+R^{-1}(0))) - g(x-R^{-1}(0))
\\&&+g(-x+R^{-1}(|J|)) - g(x-R^{-1}(|J|)) 
\\&&+
\chi_{\{1\}}(j) g(1-x)- \chi_{\{1\}}(j) g(x)
\end{eqnarray*} 
is a  bounded function, whose first and second derivatives are bounded. Since the derivatives of the difference is bounded, the Birkhoff sums are $O(N)$, and so it suffices to prove claims (3), (4) and (5) for this function.  
If $x\in G(M)$, Corollary \ref{C:sums} gives that the Birkhoff sums of the function
\begin{multline*}
g(-x+R^{-1}(0)) - g(x-R^{-1}(0))+\\
g(-x+R^{-1}(|J|)) - g(x-R^{-1}(|J|))
\end{multline*} 
and its derivative are sufficiently small that they may be ignored.    Indeed, the consecutive terms have that the $N\log(N)$ terms appear with opposite signs and by our assumption on the set $G(M)$ the other terms are $o(N\log(N))$.

So it remains to consider Birkhoff sums of 
$$\chi_{\{1\}}(j) g(-x)-\chi_{\{1\}}(j) g(x).$$
If $p=(x,j)\in G(M)$, then the orbit of $p$ of length $2M$ stays in either $U_{\ell+1}$ or $V_{\ell+1}$.
By Lemma \ref{L:insignificant1} it suffices to study Birkhoff sums of 

$$(\chi_{\{1\}}(j) g(1-x)- \chi_{\{1\}}(j)g(x))\chi_{(0, \fp{q_{n_{\ell-1}}\alpha}]\cup[1-\fp{q_{n_{\ell-1}}\alpha},1)}(x).$$

Suppose $p\in V_\ell$. Recall that Theorem \ref{T:skew} asserts that $V_\ell$ contains all of $[0, \fp{q_{n_{\ell}}\alpha})\times \{1\}$, 
 and all of $[0, \fp{q_{n_{\ell-1}}\alpha})\times \{1\}$ except a set $S$ that is the orbit of length $q_{n_\ell}$ 
  of an interval of size $\fp{q_{n_{\ell}}\alpha}$. Hence because $B_{S^1}(M)\supset \cup_{i=-q_{n_\ell}}^{\sqrt{a_{n_\ell+1}}q_{n_\ell}}R^{-i}(J'_\ell)$, by Lemma \ref{L:insignificant2} and Corollary \ref{C:sums}
$$\sum_{i=0}^{N-1}\chi_{[0, \fp{q_{n_{\ell-1}}\alpha})\times \{1\}}g(T^ip)=- N\log(N)+o(N\log N).$$
Similarly one obtains, 
$$\sum_{i=0}^{N-1}\chi_{[1-\fp{q_{n_{\ell-1}}\alpha},1)\times \{1\}}g(T^ip)= o(N\log N)$$
 as follows. $V_\ell$ is disjoint from $[1-\fp{q_{n_{\ell-1}}\alpha},1)\times \{1\}$ except an orbit of length $q_{n_\ell}$ 
 of an interval of size $\fp{q_{n_{\ell}}\alpha}$. Furthermore, by Lemma \ref{L:dis}, this orbit of length $q_{n_\ell}$ 
 of an interval of size $\fp{q_{n_{\ell}}\alpha}$ is disjoint from $[1-\fp{q_{n_{\ell}}\alpha},1)\times \{1\}$. So by a corresponding version of Lemma \ref{L:insignificant2} for the function $\hat{g}(x)=\frac 1 {1-x}$ the estimate follows.  
This gives Claim (3) for $p\in V_\ell$. The case $p\in U_\ell$ is similar. 

\bold{Claim (4):} The proof of claim (4) is very similar to that of claim (3). 

\bold{Claim (5):} It suffices to bound the Birkhoff sums of 
\begin{eqnarray*}
&&g(-x+R^{-1}(0))) - g(x-R^{-1}(0))
\\&+&g(-x+R^{-1}(|J|)) - g(x-R^{-1}(|J|)) 
\\&+& g(1-x)- g(x)
\end{eqnarray*} 
which is strictly larger than $|f|$. Thus Corollary \ref{C:sums} gives the bound, since the closest hit term contributes at most $M (\log M)^{\frac1{12}}$.

\bold{Claim (6):} This follows because $B(M)$ contains all points that orbit into a discontinuity in $2M+5\sqrt{\log(M)}$ iterates of $T$.
\end{proof}

\section{Unique ergodicity}\label{sec:ue}

\begin{prop}\label{prop:suff holds}
$T$ is uniquely ergodic. 
\end{prop}

It is likely that this theorem can be derived from work of Trevi\~{n}o \cite[Theorem 3]{trev}. (To do this one would find a flat surface $\omega$ where $T$ arises as a first return map of the vertical flow. One would then determine the systoles of $g_t\omega$.) 
 We apply a different approach.

Define $$S_k=\bigcup_{i=0}^{q_{n_k-1}-1} R^i(J_k')\times \bZ_2.$$

By Assumption \ref{A:2}, $$\frac 1 3 q_{n_{k}}=\frac 1 3(2q_{n_{k}-1}+q_{n_{k}-2})<q_{n_k-1}<\frac 1 2 q_{n_k}.$$ 
If $p\in S_k$, then  $p$ is either in $U_k\cap V_{k-1}$ or $V_k\cap U_{k-1}$, and by the last inequality and Theorem \ref{T:skew} part 2, the orbit of $p$ for time at least $q_{n_k-1}$ remains either exclusively in $U_k\cap V_{k-1}$ or exclusively in $V_k\cap U_{k-1}$. 

Recall that the length of $J_k'$ is 
$$\fp{q_{n_k}\alpha} \geq \frac{1}{2q_{{n_k}+1}}  \geq \frac{1}{3q_{{n_k}}a_{n_{k}+1}}.$$ By Assumption \ref{A:infsum} and the inequality above, this gives that the sum of the measures of the $S_k$ is infinite.  

Let $\lambda_{S^1}$ denote Lebesgue measure on $S^1$.
\begin{lem}\label{lem:hits right} There exists $C>0$ such that if $L \neq k$
$$\lambda(S_k\cap S_L)\geq C \lambda(S_k)\lambda(S_L).$$
\end{lem}

\begin{proof} 
Without loss of generality, we assume $L>k$. 
By definition, the projection of $S_k$ to $S^1$ is an orbit of length  $q_{n_{k}-1}$ of an interval of size $\fp{q_{n_{k}}\alpha}$, and the projection of $S_L$ 
  is an orbit of length $q_{n_{L}-1}$ of an interval of size $\fp{q_{n_{L}}\alpha}$. 

Thus it suffices to show that there is some $C>0$ such that if $I_k$ is an interval of size $\fp{q_{n_{k}}\alpha}$, then for any $x$, 
$$\sum_{i=0}^{q_{n_{L}-1}-1} \chi_{I_k}(R^i(x)) \geq C q_{n_{L}-1} \lambda_{S^1}(I_k).$$
The desired result then follows by summing as $I_k$ ranges over the intervals of $S_k$, and integrating $x$ over an interval of size $\fp{q_{n_{L}}\alpha}$.

We now prove the sufficient condition. By Denjoy-Koksma, the sum is within $2$ of $q_{n_{L}-1} \lambda_{S^1}(I_k)$. Observing 
\begin{eqnarray*}
\frac{q_{n_{L}-1}\lambda_{S^1}(I_k)-2 }{q_{n_{L}-1} \lambda_{S^1}(I_k)} 
&\geq&1- \frac{1}{q_{n_{L}-1}\fp{q_{n_{k}}\alpha}} 
\\&\geq& 1  - \frac{2q_{n_{k}+1}}{q_{n_{L}-1}} 
\\&\geq&1 - \frac{6q_{n_{k}+1}}{q_{n_{k+1}}} \to 1
\end{eqnarray*}
gives the result. In the last line, we used $L>k$, $q_{n_L-1}>\frac 1 3 q_{n_L}$ and Assumption \ref{A:rapid up}. 
\end{proof}

\begin{prop}\label{prop:ue suff} To prove that $T$ is uniquely ergodic, it suffices to show for $\lambda$ almost every $x$ we have that $(x,i)$ is in $S_k$ for infinitely many $k$.
\end{prop}

This will be proved by showing that ergodic measures are absolutely continuous with respect to Lebesgue and points that are in infinitely many $S_k$  cannot be generic for an ergodic measure unless the ergodic measure is Lebesgue.

\begin{lem}\label{lem:erg decomp}If $T$ is not uniquely ergodic then there exist exactly two ergodic probability measures $\mu,\nu$ with $\mu=\iota(\nu)$, and $\lambda=\mu+\nu$, and both $\mu$ and $\nu$ are  absolutely continuous with respect to Lebesgue.
\end{lem}

\begin{proof}
Suppose $\nu$ is an ergodic measure that is not Lebesgue.  Since $\iota$ commutes with $T$, $\mu=\iota(\nu)$ must also be ergodic. If $\mu=\nu$, then $\mu$ is $\iota$ invariant and hence must be Lebesgue, since $R$ is uniquely ergodic. Since $R$ is uniquely ergodic, $\nu$ and $\mu$ both project to Lebesgue. 

Since $\mu+\nu$ is $\iota$ invariant and $R$ is uniquely ergodic, $\lambda=\mu+\nu$. Hence $\mu$ and $\nu$ are absolutely continuous with respect to Lebesgue. 

If $T$ had a third ergodic measure $\mu'$, then $\mu'=\iota(\nu')$ would also be ergodic, and we'd have $\mu'+\nu'=\mu+\nu$. This contradicts uniqueness of  ergodic decompositions. 
\end{proof}

\begin{lem}\label{L:Uo}
$U_k$ is the union of $o(q_{n_{k+1}})$ disjoint intervals. 
\end{lem}

\begin{proof}
By the inductive definition of $U_k$, it is clear that $U_k$ is the union of at most $O\left( \sum_{i=1}^k q_{n_i}  \right)$ disjoint intervals. Thus it suffices to show that $\sum_{i=1}^k q_{n_i}$ is $o(q_{n_{k+1}})$. To do so, note the following crude estimate, 
$$q_{n_{k+1}}> q_{n_{k}+1}>a_{n_{k}+1}q_{n_k}.$$ 
Note also that by Assumption \ref{A:toavoid0}, 
$$\sum_{i=1}^{k} q_{n_i} = q_{n_k}+\sum_{i=1}^{k-1} q_{n_i}<2q_{n_k}.$$ Thus Assumption \ref{A:growing} gives the result. 
\end{proof}

\begin{lem}Let $A\subset S^1\times \bZ_2$ be any measurable set. For any $\epsilon>0$,
$$\lim_{k\to\infty} \lambda\left(\left\{(x,j)\in U_{k}: \left|\frac 1 {q_{n_{k}}} \sum_{i=0}^{q_{n_{k}}-1}\chi_A(T_{k}^i(x,j))-\lambda(A\cap U_{k})\right|>\epsilon\right\}\right)=0.$$ The same result is true with $U_{k}$ replaced by $V_{k}$.
\end{lem}
Note that $T_k$ appears in the statement, not $T$.
\begin{proof}
It suffices to show this statement for $A$ an interval, since $A$ can be approximated by a union of intervals. 

 Let $A_k$ be the projection of $A\cap U_{k}$ to $S^1$. Since $U_k$ is $T_k$ invariant and projects bijectively to $S^1$, the lemma is equivalent to  
 $$\lim_{k\to\infty} \lambda\left(\left\{x\in S_1: \left|\frac 1 {q_{n_{k}}} \sum_{i=0}^{q_{n_{k}}-1}\chi_{A_k}(R^i(x))-\lambda(A_k)\right|>\epsilon\right\}\right)=0.$$
 
 Note that since the measure of the symmetric difference of $U_k$ and $U_{k-1}$ goes to zero,
  the measure of the symmetric difference of $A_k$ and $A_{k-1}$ must also go to zero. 
 
 Thus $|\chi_{A_k}-\chi_{A_{k-1}}|$ has  $L^1$ norm going to zero, so the set of points where the Birkhoff sums for $\chi_{A_k}$ and $\chi_{A_{k-1}}$ at time $q_{n_k}$ differ by more than $\epsilon/2$ goes to zero (simply because a function with small $L^1$ norm can't be big very often: $\lambda(\{x:f(x)>C\})<\frac{\|f\|_1}C$). Hence, it suffices to show the equivalent result with $A_k$ replaced by $A_{k-1}$ and $\epsilon$ replaced by $\epsilon/2$.
 
 By the previous lemma, the set $A\cap U_{k-1}$  is a disjoint union of $o(q_{n_k})$  disjoint intervals, and hence $\chi_{A_{k-1}}$ has total variation  $o(q_{n_k})$. The statement is now implied by Denjoy-Koksma for the function $\chi_{A_{k-1}}$. 
 \end{proof}
The  proof of Proposition \ref{prop:ue suff} will consist of two main steps.  The first is  to show that if $\mu$ is an ergodic measure other than $\lambda$ it is without loss of generality the weak-* limit of uniform measure on the $U_k$. The second step is to show that if $(x,i)$ is in  $S_k$ for infinitely many $k$ it can not be a generic point  for $\mu$.
\color{black}
\begin{proof}[Proof of Proposition \ref{prop:ue suff}] Assume that $\lambda$ is not uniquely ergodic. Then Lemma \ref{lem:erg decomp} gives two ergodic measures $\mu$ and $\nu$.  Let $A$ be a $T$ invariant set such that  $\mu(A)=1$ and $\nu(A)=0$. Let $\epsilon>0$ be arbitrary. By the previous lemma, we can find a $k_0$ such that  for all $k\geq k_0$ we have
$$\lambda\left(\left\{(x,j)\in U_{k}: \left|\frac 1 {q_{n_{k}}} \sum_{i=0}^{q_{n_{k}}-1}\chi_A(T_{k}^i(x,j))-\lambda(A\cap U_{k})\right|>\epsilon\right\}\right)<\epsilon.$$
Using a measure estimate as in the proof of Claim 1 of Theorem \ref{thm:output jon}, we see that, possibly after increasing $k_0$, we may assume that for all $k\geq k_0$, 
$$\lambda(\{(x,j):T_k^i(x,j)= T^i(x,j) \text{ for all }0\leq i\leq q_{n_k}\})>1-\epsilon.$$ 
Because $A$ is $T$ invariant we have that for all $r$  $$\frac 1 {r} \sum_{i=0}^{r-1}\chi_A(T^i(x,j))$$ is almost everywhere either 0 or 1. Keeping in mind that since $\lambda=\mu+\nu$ we have $\lambda(A)=1$, it follows from the previous lemma that for all  $k\geq k_0$ we have 
$$\lambda(U_k\cap A)>1-2\epsilon \text{ and } \lambda(V_k \cap A)<2\epsilon$$ or 
$$\lambda(U_k\cap A)<2\epsilon \text{ and }\lambda(V_k \cap A)>1-2\epsilon.$$ Since $\underset{k \to \infty}{\lim} \, \lambda(U_{k-1} \Delta U_k)=0$ the property that $\lambda(U_k \cap A)$ is almost 0 or almost 1 is eventually constant (in $k$). Without loss of generality, suppose $\lambda(U_k\cap A)>1-2\epsilon$ for all large enough $k$.  Since this is true for all $\epsilon>0$, and since $\mu$ projects to Lebesgue, it follows that $\mu$ is the weak-* limit of uniform measure on the $U_k$.

Let $p=(x,i)$ be a point that is $\mu$ generic and that is contained in infinitely many $S_\ell$. Thus there are infinitely many times $\ell$ for which the orbit  of $p$ is disjoint from $U_\ell$ for time $q_{n_{\ell+1}-1}$. (As we remarked at the beginning of this section, this is the case when $p\in S_{\ell+1}\cap V_\ell \cap U_{\ell+1}$. Note the ``index shift" by one: to get disjointness from $U_\ell$ for a long time, we use points in $S_{\ell+1}$. )

By the existence of a density point for the set $A$, 
 there must be some interval $I\in S_1\times \bZ_2$ such that $\mu(I) \geq 0.99 \lambda(I)$. Hence for large $k$, $V_k$ contains at most $0.1$ of the $\lambda$ measure of $I$. The projection of $V_k\cap I$ to $S^1$ thus has measure at most $0.1 \lambda(I)$, and consists of $o(q_{n_{k+1}})$ intervals (Lemma \ref{L:Uo}). Now assume the orbit segment of length of $q_{n_{k+1}-1}$ of $(x,i)$ is disjoint from $U_k$. Because $q_{n_{k+1}-1}>\frac 1 3 q_{n_{k+1}}$ by Denjoy-Koksma, the orbit of $x$ up to time $q_{n_{k+1}-1}$ spends at most  $0.2 \lambda(I)$ of its time in the projection of $I$ to $S_1$. 

Hence $$\underset{N \to \infty}{\liminf} \frac 1 N \sum_{i=0}^{N-1} \chi_{I}(T^ip)\leq .2\lambda(I),$$ since $\sum_{i=0}^{N-1} \chi_{I}(T^ip)$ is  
bounded by a Birkhoff sum of the characteristic function of the projection of $I\cap V_k$ to $S^1$.

 Thus for $p$, the liminf of the Birkhoff sums of $\chi_I$ is at most $0.2\lambda(I)$, which is a contradiction to Birkhoff's ergodic theorem, and the facts that $p$ is $\mu$ generic and $\mu(I) \geq 0.99 \lambda(I)$. 
\end{proof}
\color{black}
\begin{proof}[Proof of Proposition \ref{prop:suff holds}] We first prove a well known and straightforward result.

\bold{Claim:} Let $A_i$ be measurable subsets of a space with a measure $\lambda$ of total mass 1. 
  If there exists $C>0$ such that $\lambda(A_i\cap A_j)>C\lambda(A_i)\lambda(A_j)$ and  $\sum_{i=1}^{\infty}\lambda(A_i)=\infty$,  then $\lambda\left( \LS A_i\right)>0$.

\bold{Proof of Claim:} Let $B_{N,M}=\cup_{i=N}^M A_i$. If $\sum_{i=N}^M\lambda(A_i)<\frac 1 {2C}$ then for any $j \notin [N,M]$ we have that
$$\lambda(A_j \setminus B_{N,M})\geq \lambda(A_j)-\sum_{i=M}^NC\lambda(A_j)\lambda(A_i)>\frac 1 2 \lambda(A_j).$$ 
  Because $\sum \lambda(A_i)=\infty$, we have $\lambda(B_{N,\infty})\geq \frac 1 {4C}$ for all $N$. Because we are in a finite measure space it follows that $\LS A_i$ has positive measure, which proves the claim.

 Now we complete the proof of the proposition. By Lemma \ref{lem:hits right}  we have  $\lambda\left(\LS S_i\right)>0$. Thus the set of points in infinitely many $S_k$ has positive Lebesgue measure. 
Note $(x,i) \in S_k$ depends only on $x$ and being in infinitely many $S_i$ is almost everywhere $R$ invariant. This is because the difference between $S_i$ and $T(S_i)$ is at most two intervals of size $\fp{q_{n_i}\alpha}$, so the difference between $\LS S_i$ and $T\left(\LS S_i\right)$ 
has size bounded by $\sum_{j=i}^{\infty}2\fp{q_{n_j}\alpha}$ 
for all $i$ and hence must be measure zero.  

 So by the ergodicity of $R$ almost every point is in infinitely many $S_i$.
\end{proof}

\section{Mixing of the suspension flow}\label{S:mix}

The purpose of this section is to prove Theorem \ref{T:genus2}, by proving that the flow over $T$ with roof function $f$ is mixing. 
This section shows that if $T$ is a multi-interval exchange transformation that is ergodic with respect to Lebesgue measure, and $f$ is an integrable function with $f(x)\geq 1$ for all $x$, and such that $f$ and $T$ satisfy Theorem \ref{thm:output jon}, then the suspension flow over $T$ with roof function $f$ is mixing. The fact that estimates like Theorem \ref{thm:output jon} are sufficient for mixing is standard, see for example the sufficient condition for mixing in \cite[Theorem 2.1]{koc_1}.

\begin{lem}
There is a set $I(M)$ of disjoint subintervals of $G(M)$ such that each interval has size between $\frac{1}{M\sqrt{\log(M)}}$ and $\frac{2}{M\sqrt{\log(M)}}$, and consists entirely of points $p$ for which 
$$\frac34 MC \leq \sum_{i=0}^{M-1} f(T^ip) \leq \frac43 M C,$$
where $C=\int f$, and such that the union of the intervals of $I(M)$ has measure going to 1. 
\end{lem}

\begin{proof}
Since $G(M)$ is the disjoint union of intervals of length at least $\frac{2}{M\log(M)}$, it can be divided into disjoint intervals of length between $\frac{1}{M\sqrt{\log(M)}}$ and $\frac{2}{M\sqrt{\log(M)}}$. Let $I_0(M)$ be the set of these disjoint intervals.

By the Birkhoff Ergodic Theorem, the set of points $p$ for which 
$$\frac45 MC \leq \sum_{i=0}^{M-1} f(T^ip) \leq \frac54 M C$$
has measure going to 1. Let $I(M)$ be the set of intervals in $I_0(M)$ that contain such a point $p$. 

The Mean Value Theorem and Theorem \ref{thm:output jon} part (3)  give that if $p$ satisfies the above bound, then all points within distance $\frac{2}{M\sqrt{\log(M)}}$ of $p$ satisfy the weaker bound in the lemma statement for $M$ large enough. 
\end{proof}

We now wish to study the curve $F^{CM}(I)$, where $I$ is a fixed interval in $I(M)$. To this end, if $p\in I$, let $N(p)$ denote the unique integer such that $$\sum_{i=0}^{N(p)-1} f(T^i p) \leq CM < \sum_{i=0}^{N(p)} f(T^i p).$$ So $N(p)$ is the number of times $F^t(p)$ returns to the base up to time $t=CM$, counting $t=0$.

We recall that by definition, in suspension flows vertical segments are orbit segments of the flow.
\begin{lem}\label{L:graph}
For $M$ large enough, for every interval $I\in I(M)$ the graph of $$\sum_{i=0}^{N(p)-1} f(T^i p)$$ lies within $o(1)$ of a vertical line segment of length between $\sqrt{\log(M)}/4$ and $5\sqrt{\log(M)}$. Moreover, this is true in a parametrized sense: the graph is piecewise $C^1$ with slope within $o(M\sqrt{\log(M)})$ of a  constant at every point.
\end{lem}

\begin{proof}  
 The previous lemma gives that 
$$\frac34 CN(p) \leq \sum_{i=0}^{N(p)-1} f(T^i p) \leq \sum_{i=0}^{N(p)} f(T^i p) \leq \frac43 C(N(p)+1),$$
 so $\frac34 M-1 \leq N(p) \leq \frac 43M$ for all $p\in I$. We will only use the weaker bounds $\frac12 M\leq N(p) \leq 2M$.

The vertical length of $F^{CM}(I)$ is 
$$\int_I  \sum_{i=0}^{N(p)-1}  f'(T^i p)dp,$$
which, using Theorem \ref{thm:output jon} and $N(p)\leq 2M$, has size at most 
\begin{eqnarray*}
\int_I  \left(\left|\sum_{i=0}^{2M}  f'(T^i p)dp\right|\right)
&\leq &2\frac{2M\log(2M)+o(2M\log(M))}{M\sqrt{\log(M)}}
\\&\leq & 5\sqrt{\log(M)}
\end{eqnarray*}
for $M$ sufficiently large. 

Similarly, $$\left|\int_I  \sum_{i=0}^{N(p)-1}  f'(T^i p)dp\right|\geq \frac{\frac 1 2 M\log(\frac 1 2 M)+o(M\log(M))}{M\sqrt{\log(M)}}>\frac{\sqrt{\log(M)}}4$$
for $M$ sufficiently large.

Now, since $F^{CM}(I)$ has length at most $5\sqrt{\log(M)}$, and the roof function $f$ is always at least 1, it follows that $F^{CM}(I)$ can hit the base at most $5\sqrt{\log(M)}$ times. That is, if $p$ and $p'$ are any two points in $I$, then $|N(p)-N(p')|\leq 5\sqrt{\log(M)}$. 

Now let $N=N(p)$ for any point $p\in I$, and let $p'\in I$ be another point. Define 
$$s=\sum_{i=0}^{N(p)-1}  f'(T^i p).$$
By Theorem \ref{thm:output jon} part (3) we have $s=\pm N\log(N) + o(N \log N)$.
By Theorem \ref{thm:output jon} part (4) and the Mean Value Theorem we have that 
$$\left| s - \sum_{i=0}^{N(p)-1}  f'(T^i p')\right|\leq  o((\log(N))^{\frac 1 3-\frac12}N)$$
and by  Theorem \ref{thm:output jon} part (5) we have 
$$\left| \sum_{i=0}^{N(p')-1}  f'(T^i p') - \sum_{i=0}^{N(p)-1}  f'(T^i p')\right|=o(N \sqrt{\log (N)}).$$
We get that the given graph is piecewise $C^1$ with slope within $o(M\sqrt{\log(M)})$ of $s$ at every point. 
This completes the proof. 
\end{proof}

Let $\Lambda$ denote 2 dimensional Lebesgue probability measure on the suspension of $T$ by $f$. Let $R$ be a rectangle, by which we mean a rectangle contained strictly under the graph of $f$. Say that a point $p$ is $(L, \epsilon, R)$ good if every vertical line $V$ through $p$ of length at least $L$ has 
$$\left| m_V(V\cap R)-\Lambda(R)\right| \leq \epsilon,$$ 
where $m_V$ is the 1 dimensional Lebesgue probability measure on $V$. 

\begin{lem}\label{lem:most good} 
For fixed $\epsilon$ and $R$, the set of points that are $(L, \epsilon, R)$ good has measure going to 1 as $L\to\infty$. 
\end{lem}

\begin{proof}
This follows from ergodicity of the flow. 
\end{proof}

\begin{lem}
Let $R$ and $R'$ be two rectangles. Then 
$$\lim_{t\to\infty} \Lambda(R\cap F^t(R'))\to \Lambda(R)\Lambda(R').$$
\end{lem}

\begin{proof}
Let $\epsilon>0$ be arbitrarily small and in particular much smaller than the height and width of $R$. Let $R_s$ be the set of points in $R$ that have distance at least  $\epsilon/50$ to the boundary of $R$, and let $R_b$ be an $\epsilon/50$ neighbourhood of $R$. (``$s$" and ``$b$" stand for ``smaller" and ``bigger".)

Let $H$ denote the $y$-coordinate of the top edge of $R'$.  Define $D_L$ to be the set of points $p$ 
 such that all $F^h(p)$ with $0\leq h\leq H$ are both $(L, \epsilon/100, R_s)$ and $(L, \epsilon/100, R_b)$ good. Pick $L$ large enough such that $D_L$ has $\Lambda$-measure at least $1-\epsilon/100$.

 Let  $M_0$ be large enough so that
 \begin{itemize}
 \item for all $M\geq M_0$ we have $\sqrt{\log(M)}/4>\max(L,H)$,
 \item the $o(1)$ error in Lemma \ref{L:graph} is less than $\epsilon/100$,
 \item the union of the intervals in $I(M)$ has $\lambda$-measure 
  at least $1-\epsilon/100$,
 \item the error in Theorem \ref{thm:output jon} parts (3) and (5) are less than $\epsilon/1000$, 
 \item $2H/(M_0\sqrt{\log M_0})<\epsilon/1000$
 \end{itemize}

 \bold{Claim.} If $t\geq 10C M_0$, $M=\lfloor \frac t C \rfloor$, $I \in I(M)$ and there exists $x \in I$ such that $F^{t}x \in D_L$, then for any $0\leq h\leq H$ we have  $$\left| \frac 1 {\lambda(I)}\int_I \chi_R (F^{t+h}y)dy-\Lambda(R)\right|<\e/10.$$ 
 \bold{Proof of claim.} Let $V_0$ be the vertical line produced by Lemma \ref{L:graph} through $F^{CM}x $ that approximates $F^{CM}(I)$  to within $\epsilon/100$.
  and the Mean Value Theorem, the vertical line $V=F^{h+t-CM}(V_0)$ approximates  $F^{CM+h+(t-CM)}(I)$ to within $\epsilon/50$, and observe
$$
m_V(R_s\cap V)\leq \frac 1 {\lambda(I)}\int_I\chi_R (F^{t+h}x)dx\leq m_V(R_b \cap V).
$$
The claim follows since $F^{t}x \in D_L$.

Now, given $p\in R'$, let $x_p$ denote its $x$-coordinate, and let $y_p$ denote its $y$-coordinate.
Let  $t>10 CM_0$ and let  $M=\lfloor \frac t C\rfloor$. Set
$$P_t=F^{-t}(D_L) \cap\{p : x_p\in J \text{ for some }J \in I(M)\}.$$
For $p\in P_t$, let $J_M(p)$ be the interval $J \in I(M)$ containing $x_p$. 

We consider 
\begin{eqnarray*}
\Lambda(R\cap F^{t}(R'))&=&\int_{p\in P_t \cap R'} \frac 1 {\lambda(J_M(p))} \int _{J_M(p)}\chi_R(F^{t+y_p}(x))dxdp\\
&+&E
\end{eqnarray*}
where the error $E$ is at most the size of $P_t^c \cap R'$ plus $2H/(M\sqrt{\log(M)})$. 
  The second part of the error term comes from when $J_M(p)$ is not entirely contained in in the projection of $R'$ to the base.
In particular, the error $E$ is at most $\epsilon/10$. 

Noting that $0\leq y_p \leq H$ and applying Fubini and the above claim for the inner integral gives the result. 
\end{proof}

Since it is enough to prove mixing for rectangles, this proves Theorem \ref{T:genus2}.

\section{A mixing flow on a surface with non-degenerated fixed points}\label{S:surface}

Define $x_\infty=\alpha-\sum_{i=1}^{\infty}\fp{q_{n_i-1}\alpha}$ and $x_k=\alpha-\sum_{i=1}^{k}\fp{q_{n_i-1}\alpha}$. We now define another IET
$$\hat{T}:S^1\times\bZ_2\times\bZ_2 \to S^1\times\bZ_2\times\bZ_2$$ by 
$$\hat{T}(x,i,j)=(R(x),i +\chi_J(R(x)),j+\chi_{J+x_\infty}(R(x))).$$
The key observation about $\hat{T}$ is that its restriction to the first and second coordinates is $T$, and its restriction to the first and third coordinates is a ``translate" of $T$ by $x_\infty$. Our choice of $x_\infty$ was motivated by the proof of Theorem \ref{T:hatUE} below, which gives that $\hat{T}$ is uniquely ergodic. Observe that Assumption \ref{A:12} implies that the intervals $J$ and $x_\infty+J$ are disjoint.

Define the roof function 
\begin{eqnarray*}
\hat{f}(x,i,j)&=&1+ |\log(d(Rx,0))|+|\log(d(Rx,|J|)|\\
&+& |\log(d(Rx,x_\infty))|+ |\log(d(Rx,|J|+x_\infty))|\\
&+&\chi_{\{1\}}(j)|\log(d(Rx,0))|.
\end{eqnarray*} 
The first two lines put  logarithmic singularities of equal weight (coefficient) over all discontinuities of $\hat{T}$, and the third line introduces additional weight to the singularity over one pair of the discontinuities. 

The purpose of this section is to show 

\begin{thm}\label{T:sec9}
The flow over $\hat{T}$ with roof function $\hat{f}$ is mixing. 
\end{thm}

At the end of this section, we will use this to conclude Theorem \ref{T:main}.

\bold{Non-minimal approximates.} Define 

$$\hat{T}_\ell(x,i,j)=(R(x), \chi_{[0, \sum_{k=1}^\ell 2\fp{q_{n_k} \alpha})}(R(x)), \chi_{[0, \sum_{k=1}^\ell 2\fp{q_{n_k} \alpha})+x_\infty}(R(x))).$$

Define the following subsets of  $S^1\times\bZ_2\times\bZ_2$,
\begin{align*}
&U_k^1=\{(x,i,j): (x,i)\in U_k\},  &V_k^1&=\{(x,i,j): (x,i)\in V_k\},\\
&U_k^2=\{(x,i,j): (x,j)\in U_k+x_\infty\},  &V_k^2&=\{(x,i,j): (x,j)\in V_k+x_\infty\}.
\end{align*}

A corollary of Lemma \ref{L:invar} is the following. 

\begin{cor}
The above four sets are all  $\hat{T}_k$ invariant. 
\end{cor}
Note that all of the sets 
$$U_k^1\cap U_k^2,\quad  U_k^1\cap V_k^2, \quad V_k^1\cap U_k^2,\quad V_k^1\cap V_k^2$$
are invariant and project bijectively to $S^1$. Since $R$ is minimal, we see that each of these four sets is minimal. Since $S^1\times\bZ_2\times\bZ_2$ is the disjoint union of these four sets, the minimal components of $\hat{T}_k$ are exactly these four sets. 

Before we continue, we note some estimates. 

\begin{lem}\label{L:kinfinity}
The following hold. 
\begin{enumerate}
\item $\fp{q_{n_\ell-1}\alpha}>4\fp{q_{n_\ell}\alpha}$.
\item $d(x_\ell,x_\infty)<\frac12 \fp{q_{n_\ell}\alpha}$.
\item If $N=1+\sum_{i=0}^{\ell} q_{n_i-1}$, then $\frac13 q_{n_\ell}\leq N< q_{n_\ell}$. 
\end{enumerate}
\end{lem}

\begin{proof}
The first claim follows from Theorem \ref{T:R}, since 
$$\fp{q_{n_\ell-1}\alpha} > \frac1{2 q_{n_\ell}} > \frac4{q_{n_\ell+1}} > 4\fp{q_{n_\ell}\alpha},$$
where we have used that, given our assumptions, $q_{n_\ell+1}> 8 q_{n_\ell}$.

 The second claim follows from the fact that 
 \begin{eqnarray*}
 d(x_\ell,x_\infty)&=&\sum_{k=\ell+1}^\infty2\left| \fp{q_{n_k-1}\alpha}\right|
 \\&\leq& \sum_{k=\ell+1}^\infty\frac{2}{q_{n_k}}\leq\frac{4}{q_{n_{\ell+1}}}\leq \frac12 \fp{q_{n_\ell}\alpha}
 \end{eqnarray*}
 and the exponential growth of the $q_{j}$ (compare to the start of proof of Theorem \ref{thm:output jon}).
 
 The upper bound in the third claim follows by noting that Assumption \ref{A:toavoid0} gives $N<q_{n_{\ell-1}}+q_{n_{\ell}-1}$. The lower bound is obtained by noting  $N\geq q_{n_\ell-1}$ and $q_{n_\ell}\leq (a_{n_\ell}+1)q_{n_{\ell}-1}$ and then using Assumption \ref{A:2}.
\end{proof}

\bold{Unique ergodicity.} Before proving unique ergodicity we need the following lemma. Recall $J_\ell'=\left[\sum_{k=1}^{\ell-1} 2\fp{q_{n_k} \alpha},\sum_{k=1}^{\ell-1}2\fp{q_{n_k} \alpha}+\fp{q_{n_\ell} \alpha}\right).$

\begin{lem}\label{L:dis hat}
For each $\ell>0$ the intervals
$$R^i(x_\infty+J'_\ell), \quad i=0,\ldots, q_{n_\ell}$$
are disjoint from each other and $[0,\frac 1 2 \fp{q_{n_{\ell}}\alpha})$ and $[1-\fp{q_{n_{\ell}}\alpha}, 1)$.

For large enough $\ell$,
$\cup_{i=0}^{q_{n_\ell-1}} R^i(J_\ell')$ is disjoint from $x_\infty+J_\ell'$. 
\end{lem}

\begin{proof}
Compare to the proof of Lemma \ref{L:dis}. We will  use  Lemma \ref{L:kinfinity} several times. 

Set  $N=1+\sum_{i=1}^{\ell}q_{n_i-1}$ and $N_0=2\sum_{k=1}^{\ell-1} q_{n_k}$.
 Note that the intervals 
$$R^i(x_\ell+J'_\ell), \quad i=0,\ldots, q_{n_\ell}$$ 
and $[0,\fp{q_{n_{\ell}}\alpha})$ and $[1-\fp{q_{n_{\ell}}\alpha}, 1)$ are contained in an orbit of length $N_0+N+2q_{n_\ell}$ of an interval of size $\fp{q_{n_\ell}\alpha}$.

By Assumption $\ref{A:toavoid0}$, $N_0<q_{n_\ell}$, so $N_0+N+2q_{n_\ell}<4q_{n_\ell}$. By Assumption \ref{A:growing}, we have $4q_{n_\ell}<a_{n_\ell+1} q_{n_\ell} < q_{n_\ell+1}$, so in particular we conclude that $N_0+N+2q_{n_\ell}< q_{n_\ell+1}$.

Hence, by the separation property, the above intervals  are disjoint. Since $x_\ell -\frac 1 2 \fp{q_{n_\ell}\alpha}<x_\infty<x_\ell$, the intervals
 $R^i(x_\infty+J'_\ell)$ are disjoint from each other and $[0,\frac 1 2 \fp{q_{n_{\ell}}\alpha})$ and $[1-\fp{q_{n_{\ell}}\alpha}, 1)$. 


We now prove the final claim. Note $x_\ell+J'_\ell=R^N(J'_\ell)$. 
  Because the intervals in $\{R^i(J'_\ell)\}_{i=0}^{q_{n_\ell}-1}$ are $\fp{q_{n_\ell-1}\alpha}-\fp{q_{n_\ell}\alpha}$ separated,
 $R^N(J'_\ell)$ is the unique element of this orbit segment within $\fp{q_{n_\ell-1}\alpha}-\fp{q_{n_\ell}\alpha}$ of $x_\ell+J'_\ell$. Since 
 $d(x_\infty,x_\ell)<\fp{q_{n_\ell-1}\alpha}-\fp{q_{n_\ell}\alpha}$, $N$ is the unique $i \in \{0,...,q_{n_\ell}-1\}$ such
 that $R^i(J'_\ell) \cap (x_\infty +J'_\ell) \neq \emptyset.$ 
 Since $N>q_{n_\ell-1}$ the final claim follows. 
\end{proof}

\begin{thm}\label{T:hatUE}
$\hat{T}$ is uniquely ergodic. 
\end{thm}

\begin{proof}
This is similar to the proof that $T$ is uniquely ergodic. We will outline the additional considerations that are required. 

Since $T$ is uniquely ergodic, $\hat{T}$ has at most two ergodic components, each of which project to Lebesgue under the projection $(x,i,j)\mapsto (x,i)$ to the first two coordinates, and also under the projection $(x,i,j)\mapsto (x,j)$ to the first and third coordinates. If there are two ergodic measures, they must be exchanged under the involutions $(x,i,j)\mapsto (x,i+1, j)$ and $(x,i,j)\mapsto(x,i,j+1)$, and hence invariant under their product $(x,i+1, j+1)$. 
In fact one of the two measures must be the weak-* limit of Lebesgue on $E_\ell=(U_\ell^1\cap V_\ell^2)\cup (V_\ell^1\cap U_\ell^2)$, and the other must be the limit of $F_\ell=(U_\ell^1\cap U_\ell^2)\cup (V_\ell^1\cap V_\ell^2)$. 

The proof now follows as for $T$, with $S_\ell$ replaced by 
$$\cup_{i=0}^{q_{n_\ell-2}-1} R^i(J_\ell')\times \bZ_2\times \bZ_2.$$
Note that, as in the beginning of Section \ref{sec:ue}, we have $\frac 1 3 q_{n_{\ell}-1}\leq q_{n_\ell-2}\leq \frac12 q_{n_\ell-1}.$ 
Thus by the previous lemma, this set consists of points $p$ such that the orbit of $p$ of length $q_{n_\ell-2}$ is entirely in 
$E_\ell\cup F_{\ell-1}$ or entirely in $F_\ell\cup E_{\ell-1}$. 
So we may repeat the last half of the proof of Proposition \ref{prop:ue suff}.
\end{proof}

\bold{Birkhoff sums.} 
We now outline the relevant changes to the proof of Theorem \ref{thm:output jon}.

To set up context, the appropriate version of Lemma \ref{L:dis} is Lemma \ref{L:dis hat}. With this, Theorem \ref{T:skew} can be modified to
\begin{thm}\label{T:skew hat}
For each integer $\ell\geq1$,  the following properties hold. 
\begin{enumerate}
\item $U_{\ell}^2$ contains
  $$[0, \frac 1 2 \fp{q_{n_{\ell}}\alpha})\times \mathbb{Z}_2 \times \{0\}\quad\text{and}\quad[1-\fp{q_{n_{\ell}}\alpha},1)\times \mathbb{Z}_2 \times \{1\}.$$
\item 
$$\pi_{S_1}(U_\ell^2\setminus U_{\ell-1}^2)=\cup_{i=0}^{q_{n_\ell-1}} R^i(J_\ell'+x_\infty).$$
\end{enumerate}
\end{thm}

We next observe that Lemma \ref{L:insignificant1} can be modified to be
\begin{lem}\label{L:insignificant1 hat}
Set $g(x)=1/x$. 
For all $N\geq q_{n_k}$,
$$ \sum_{i=0}^{N-1} g(R^ix)\chi_{[\frac 1 2 \fp{q_{n_{k-1}}\alpha},1)}(R^ix) = o(N\log(N)) $$
and 
$$ \sum_{i=0}^{N-1} g'(R^ix)\chi_{[\frac 1 2 \fp{q_{n_{k-1}}\alpha},1)}(R^ix) = O(N^2).$$
\end{lem}

With Lemma \ref{L:dis hat} in mind, we can modify Lemma \ref{L:insignificant2} to be

\begin{lem}\label{L:insignificant2 hat}
Set $g(x)=1/x$. Let $S$ be an orbit of length $q_{n_\ell}$ of an interval of length $\fp{q_{n_{\ell}}\alpha}$. Assume $S$ is disjoint from $[0, \frac 1 2 \fp{q_{n_{\ell}}\alpha})$. Let $x$ be any point disjoint from $\cup_{i=0}^{\sqrt{a_{n_{\ell}+1}} q_{n_{\ell}}}R^{-i}(S)$. Then for all $N>q_{n_\ell}$
$$ \sum_{i=0}^{N-1} g(R^ix)\chi_{S}(R^ix) =
o(N\log(N))$$
and 
$$ \sum_{i=0}^{N-1} g'(R^ix)\chi_{S}(R^ix) =
o(N^2\log(N)^\frac13).$$
The implied constant does not depend on $\ell$. 
\end{lem}
We now outline the straightforward modifications to Theorem \ref{thm:output jon}. Let $\hat{Q}_M$ be the points within distance $\frac 1 {M(\log(M))^{\frac 1 {12}}}$ of the singularities of $\hat{f}$. 
Let 
\begin{eqnarray*}
\hat{B}_{S^1}(M)&=&\bigcup_{i=0}^{2M+5\sqrt{\log(M)}} R^{-i}(\hat{Q}_M)
\\&\cup& \bigcup_{i=0}^{2M+5\sqrt{\log(M)}} R^{-i}(x_{\infty}+ \cup_{i=\ell+1}^\infty J_\ell)
\\&\cup& \bigcup_{i=-q_{n_\ell}}^{\sqrt{a_{n_\ell+1}}q_{n_\ell}}R^{-i} (x_\infty+J_\ell').
\end{eqnarray*} 

The changes to the proof of Claims 1 and 2 are obvious changes to the measure estimates.

Claim 3 requires the most substantive changes. We assume $p \in V_\ell^2$. 
 As before, because of cancellations and our choice of $\hat{B}_{S^1}$ we may restrict our attention to  $\chi_{\{1\}}(j)(g(-x)-g(x))$. Recall $g(x)=\frac 1 {\fp{x}}$. 
By Lemma \ref{L:insignificant1 hat} we may restrict to $[0,\frac 1 2 \fp{q_{n_{\ell-1}}\alpha})\times \bZ_2\times \{1\}$ which by Theorem \ref{T:skew hat} is entirely contained in $V_{\ell-1}^{ 2}$. Then in place of Lemma \ref{L:insignificant2} we invoke Lemma \ref{L:insignificant2 hat}.

 Claim 4 is analogous to Claim 3.

The change to Claim 5 is straightforward. 

Claim 6 is straightforward.

This proves the analogue of Theorem \ref{thm:output jon}. Since we have already proven unique ergodicity of $\hat{T}$ (Theorem \ref{T:hatUE}), mixing for the suspension flow now follows as in the previous section for $T$ and $f$.

\bold{A flow on a surface.} Note that the suspension flow over $\hat{T}$ with roof function the constant function 1 is the vertical flow on the surface in Figure \ref{F:FourTori}. The first return map to the union of the four intervals at the bottom of the four parallelograms is $\hat{T}$. 

\begin{figure}[h!]
\includegraphics[width=\linewidth]{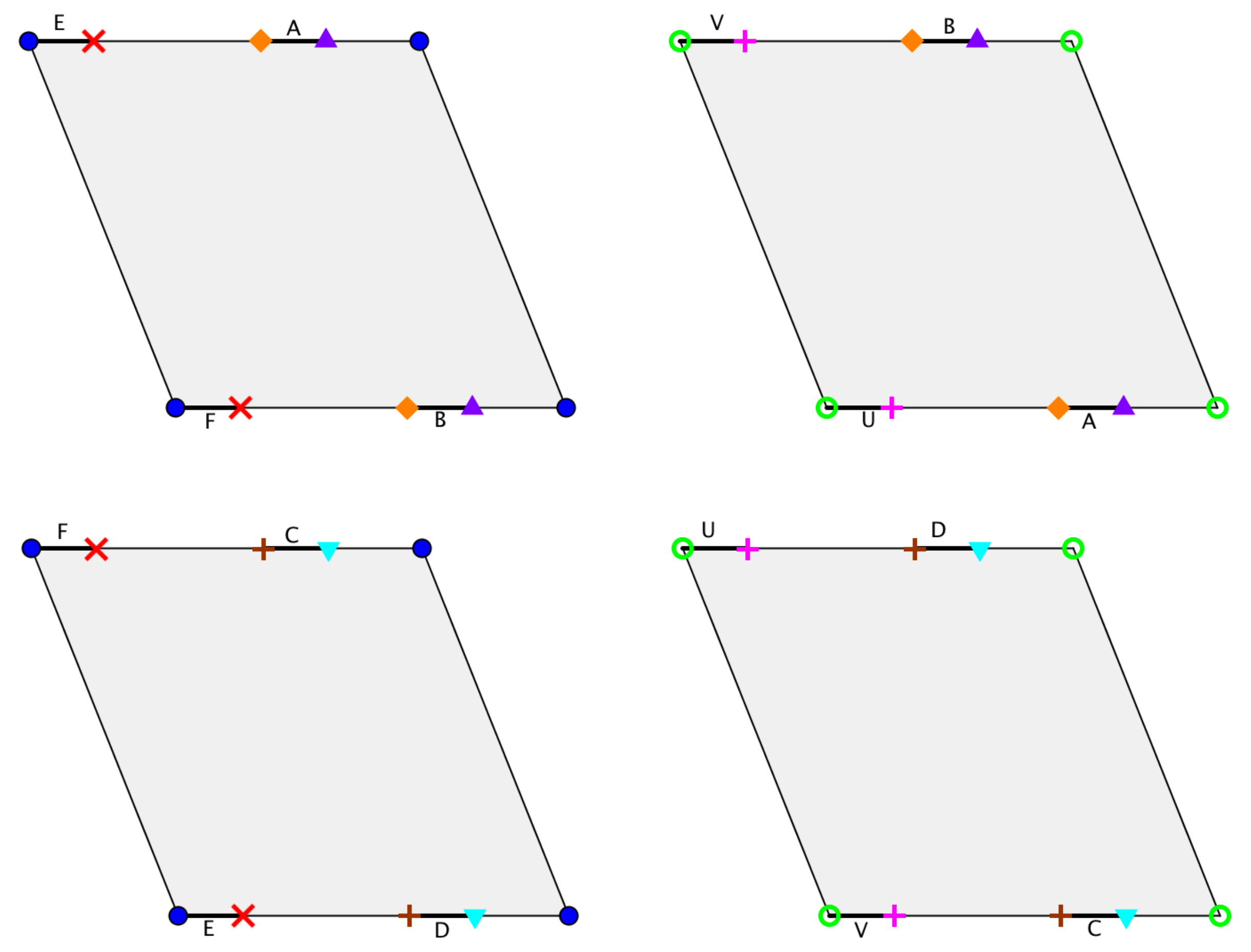}
\caption{Unmarked opposite sides are identified. The height and base of each parallelogram is one, and the shear is by $\alpha$. All of the labelled intervals have length $|J|$, and in each parallelogram the second interval is equal to the first translated by $x_\infty$.}
\label{F:FourTori}
\end{figure}


This flat surface has 8 cone points, each with angle $4\pi$. This flow is $C^{\infty}$ away from the cone points. By appropriately slowing down the flow near these fixed points--a standard procedure explained in detail in \cite[Section 7]{con_fra}--one can obtain a $C^{\infty}$ flow on this surface that has non-degenerated fixed points at the 8 distinguished points points, and such that the first return time function $h$ satisfies that $h-\hat{f}$, $h'-\hat{f}'$, $h''-\hat{f}''$ are bounded.

Because Birkhoff sums of a bounded function over orbit segments of length $N$ are $O(N)$, all estimates in this paper hold with $\hat{f}$ replaced with $h$. Hence the above arguments show that the $C^{\infty}$ flow we have produced is mixing. 

A saddle connection is a trajectory of the flow that connects singularities of the flat surface.  By the definition of the skew product,  points in $S^1\times \bZ_2\times \bZ_2$ with the same $S^1$ coordinate cannot have that their forward orbits intersect. So to show that there does not exist a saddle connection, it suffices to show that the forward $\hat{T}$ orbits each element of $\{x_\infty-\alpha, x_\infty+|J|-\alpha,-\alpha,1-\alpha+|J|\}\times \mathbb{Z}_2\times \mathbb{Z}_2$ are infinite and distinct. 
This is straightforward to check that from the construction.
 This verifies that the flow does not have saddle connections and completes the proof of Theorem \ref{T:main}.


\bibliography{mybib}{}
\bibliographystyle{amsalpha}
\end{document}